\documentclass[12pt,reqno]{amsart}

\usepackage{times}
\usepackage{amsmath,amsfonts, amstext,amssymb,amsbsy,amsopn,amsthm}
\usepackage{dsfont}
\usepackage{esint}
\usepackage{graphicx}   % for figures
\usepackage{hyperref}
\usepackage[all]{xy}
\usepackage{color}
\usepackage{tikz}
\usepackage{lineno,hyperref}

\newtheorem{thm}{Theorem}[section]
\newtheorem{cor}[thm]{Corollary}

\newtheorem{lem}[thm]{Lemma}

\theoremstyle{definition}

\theoremstyle{remark}
\newtheorem{rem}[thm]{Remark}
\theoremstyle{conclusion}

\theoremstyle{conjecture}

\numberwithin{equation}{section}

\numberwithin{equation}{section}
\allowdisplaybreaks[4]

\begin{document}

\title[Static Schr\"{o}dinger-Hartree and Schr\"{o}dinger-Maxwell equations]{Classification of nonnegative solutions to static Schr\"{o}dinger-Hartree and Schr\"{o}dinger-Maxwell equations with combined nonlinearities}

\author{Wei Dai$^\dag$, Zhao Liu$^\ddag$$^*$}

\address{$^\dag$ School of Mathematics and Systems Science, Beihang University (BUAA), Beijing 100083, P. R. China, and LAGA, Université Paris 13 (UMR 7539), Paris, France}
\email{weidai@buaa.edu.cn}

\address{$^\ddag$School of Mathematics and Computer Science, Jiangxi Science and Technology Normal University, Nanchang 330038, P. R. China, and Department of Mathematics, Yeshiva University, New York, NY, USA}
\email{liuzhao@mail.bnu.edu.cn}

\thanks{$^*$ Corresponding author: Zhao Liu at liuzhao@mail.bnu.edu.cn. \\
Wei Dai was supported by the NNSF of China (No. 11501021), the Fundamental Research Funds for the Central Universities and the State Scholarship Fund of China (No. 201806025011); Zhao Liu was supported by the NNSF of China (No. 11801237) and the State Scholarship Fund of China (No. 201808360005).}

\begin{abstract}
In this paper, we are concerned with static Schr\"{o}dinger-Hartree and Schr\"{o}dinger-Maxwell equations with combined nonlinearities. We derive the explicit forms for positive solution $u$ in the critical case and non-existence of nontrivial nonnegative solutions in the subcritical cases (see Theorem \ref{Thm0} and \ref{Thm1}). The arguments used in our proof is a variant (for nonlocal nonlinearity) of the direct moving spheres method for fractional Laplacians in \cite{CLZ}. The main ingredients are the variants (for nonlocal nonlinearity) of the maximum principles, i.e., \emph{Narrow region principle} (Theorem \ref{Thm2} and \ref{Thm3}).
\end{abstract}
\maketitle {\small {\bf Keywords:} Fractional Laplacians; Schr\"{o}dinger-Hartree equations; Schr\"{o}dinger-Maxwell equations; Nonnegative solutions; Nonlocal nonlinearities; Direct method of moving spheres.\\

{\bf 2010 MSC} Primary: 35R11; Secondary: 35B06, 35B53.}

\section{Introduction}

In this paper, we first consider the following static Schr\"{o}dinger-Hartree equation with combined nonlinearities
\begin{equation}\label{PDEH}\\\begin{cases}
(-\Delta)^{\frac{\alpha}{2}}u(x)=c_{1}\Big(\frac{1}{|x|^{2\alpha}}\ast|u|^{2}\Big)u^{p_{1}}(x)+c_2u^{p_{2}}(x), \,\,\,\,\,\,\,\, x\in\mathbb{R}^{n}, \\
u(x)\geq0, \,\,\,\,\,\,\, x\in\mathbb{R}^{n},
\end{cases}\end{equation}
where $0<\alpha\leq2$, $n\geq2$, $n>2\alpha$, $c_1,c_2\geq0$ with $c_{1}+c_{2}>0$, $0<p_{1}\leq1$ and $0<p_{2}\leq\frac{n+\alpha}{n-\alpha}$. We assume $u\in C^{1,1}_{loc}\cap\mathcal{L}_{\alpha}(\mathbb{R}^{n})$ if $0<\alpha<2$, and $u\in C^{2}(\mathbb{R}^{n})$ if $\alpha=2$, where
\begin{equation}\label{0-1}
  \mathcal{L}_{\alpha}(\mathbb{R}^{n}):=\Big\{u:\mathbb{R}^{n}\rightarrow\mathbb{R}\,\big|\,\int_{\mathbb{R}^{n}}\frac{|u(y)|}{1+|y|^{n+\alpha}}dy<\infty\Big\}.
\end{equation}
The nonlocal fractional Laplacians $(-\Delta)^{\frac{\alpha}{2}}$ with $0<\alpha<2$ are defined by (see \cite{CFY,CLL,CLZ,S,ZCCY})
\begin{equation}\label{nonlocal defn}
  (-\Delta)^{\frac{\alpha}{2}}u(x)=C_{\alpha,n} \, P.V.\int_{\mathbb{R}^n}\frac{u(x)-u(y)}{|x-y|^{n+\alpha}}dy:=C_{\alpha,n}\lim_{\epsilon\rightarrow0}\int_{|y-x|\geq\epsilon}\frac{u(x)-u(y)}{|x-y|^{n+\alpha}}dy
\end{equation}
for functions $u\in C^{1,1}_{loc}\cap\mathcal{L}_{\alpha}(\mathbb{R}^{n})$, where the constant $C_{\alpha,n}=\left(\int_{\mathbb{R}^{n}}\frac{1-\cos(2\pi\zeta_{1})}{|\zeta|^{n+\alpha}}d\zeta\right)^{-1}$.

In recent years, there has been a great deal of interest in using the fractional Laplacians to model diverse physical phenomena, such as anomalous diffusion and quasi-geostrophic flows, turbulence and water waves, molecular dynamics and relativistic quantum mechanics of stars. However, the non-local feature of the fractional Laplacians makes it difficult to study. In order to overcome this difficulty, Chen, Li and Ou \cite{CLO} developed the method of moving planes in integral forms. Subsequently, Caffarelli and Silvestre \cite{CS} introduced an extension method to overcome this difficulty, which reduced this nonlocal problem into a local one in higher dimensions. This extension method provides a powerful tool and leads to very active studies in equations involving the fractional Laplacians, and a series of fruitful results have been obtained (see \cite{Br,CZ} and the references therein).

In \cite{CLL}, Chen, Li and Li developed a direct method of moving planes for the fractional Laplacians (see also \cite{DFQ}). Instead of using the extension method of Caffarelli and Silvestre \cite{CS}, they worked directly on the non-local operator to establish strong maximum principles for anti-symmetric
functions and narrow region principles, and then obtained classification and Liouville type results for nonnegative solutions. The direct method of moving planes introduced in \cite{CLL} has been applied to study more general nonlocal operators with general nonlinearities (see \cite{CLLG,DFQ}). The methods of moving planes was initially invented by Alexanderoff in the early 1950s. Later, it was further developed by Serrin \cite{S}, Gidas, Ni and Nirenberg \cite{GNN,GNN1}, Caffarelli, Gidas and Spruck \cite{CGS}, Chen and Li \cite{CL}, Li and Zhu \cite{LZ}, Lin \cite{Lin}, Chen, Li and Ou \cite{CLO}, Chen, Li and Li \cite{CLL}, Dai and Qin \cite{DQ0} and many others. For more literatures on the classification of solutions and Liouville type theorems for various PDE and IE problems via the methods of moving planes or spheres, please refer to \cite{CD,CFY,CL1,CLZ,CY,DFHQW,DL,DLL,DQ,FC,LD,MZ,WX} and the references therein.

Chen, Li and Zhang introduced in \cite{CLZ} another direct method - the method of moving spheres on the fractional Laplacians, which is more convenient than the method of moving planes. The method of moving spheres was initially used by Padilla \cite{Pa}, Chen and Li \cite{CL2} and Li and Zhu \cite{LZ}. It can be applied to capture the explicit form of solutions directly rather than going through the procedure of deriving radial symmetry of solutions and then classifying radial solutions. In a recent work \cite{DQ0}, Dai and Qin developed the method of scaling spheres, which is essentially a frozen variant of the method of moving spheres and becomes a powerful tool in deriving asymptotic estimates for solutions. The method of scaling spheres can be applied to various fractional or higher order problems without translation invariance or in the cases Kelvin transforms in conjunction with the method of moving planes do not work (see \cite{DQ0,DQ1,DQZ} and the references therein).

When $c_2=0$ and $p_{1}=1$, PDEs of type \eqref{PDEH} arise in the Hartree-Fock theory of the nonlinear Schr\"{o}dinger equations (see \cite{LS}). The solution $u$ to problem \eqref{PDEH} is also a ground state or a stationary solution to the following $\dot{H}^{\frac{\alpha}{2}}$-critical focusing dynamic Schr\"{o}dinger-Hartree equation
\begin{equation}\label{Hartree}
  i\partial_{t}u+(-\Delta)^{\frac{\alpha}{2}}u=c_{1}\left(\frac{1}{|x|^{2\alpha}}\ast|u|^{2}\right)u, \,\,\,\,\,\,\,\,\, (t,x)\in\mathbb{R}\times\mathbb{R}^{n}.
\end{equation}
The Schr\"{o}dinger-Hartree equations have many interesting applications in the quantum theory of large systems of non-relativistic bosonic atoms and molecules (see, e.g. \cite{FL}). Dynamic equations of the type \eqref{Hartree} have been quite extensively studied, please refer to \cite{LMZ,MXZ3} and the references therein. The ground state solution can be regarded as a crucial criterion or threshold for global well-posedness and scattering in the focusing case. Therefore, the classification of solutions to \eqref{PDEH} plays an important and fundamental role in the study of the focusing Schr\"{o}dinger-Hartree equations \eqref{Hartree}.

There are lots of literatures on the qualitative properties of solutions to Hartree and Choquard equations of fractional or higher order, please see e.g. Cao and Dai \cite{CD}, Chen and Li \cite{CL0}, Dai, Fang, et al. \cite{DFHQW}, Dai and Qin \cite{DQ}, Lieb \cite{Lieb}, Lei \cite{Lei}, Liu \cite{Liu}, Moroz and Schaftingen \cite{MS}, Ma and Zhao \cite{MZ}, Xu and Lei \cite{XL} and the references therein. Liu proved in \cite{Liu} the classification results for positive solutions to \eqref{PDEH} with $\alpha=2$, $c_{2}=0$ and $p_{1}=1$, by using the idea of considering the equivalent systems of integral equations instead, which was initially used by Ma and Zhao \cite{MZ}. In \cite{CD}, Cao and Dai considered the differential equations directly and classified all the positive $C^{4}$ solutions to the $\dot{H}^{2}$-critical bi-harmonic equation \eqref{PDEH} with $\alpha=4$ and $c_{2}=0$. They also derived Liouville theorem in the subcritical cases. For general $0<\alpha<\frac{n}{2}$, Dai, Fang, et al. \cite{DFHQW} classified all the positive $H^{\frac{\alpha}{2}}(\mathbb{R}^{n})$ weak solutions to \eqref{PDEH} with $c_{2}=0$ and $p_{1}=1$ by using the method of moving planes in integral forms due to Chen, Li and Ou \cite{CLO,CLO1}. They also classified all the $L^{\frac{2n}{n-\alpha}}(\mathbb{R}^{n})$ integrable solutions to the equivalent integral equations. For $0<\alpha<\min\{2,\frac{n}{2}\}$, Dai, Fang and Qin \cite{DFQ} classified all the $C^{1,1}_{loc}\cap\mathcal{L}_{\alpha}$ solutions to \eqref{PDEH} with $c_{2}=0$ and $p_{1}=1$ by applying a variant (for nonlocal nonlinearity) of the direct method of moving planes for fractional Laplacians. The qualitative properties of solutions to general fractional order or higher order elliptic equations have also been extensively studied, for instance, see Chen, Fang and Yang \cite{CFY}, Chen, Li and Li \cite{CLL}, Chen, Li and Ou \cite{CLO}, Caffarelli and Silvestre \cite{CS}, Chang and Yang \cite{CY}, Dai and Qin \cite{DQ,DQ0,DQ1,DQZ}, Fang and Chen \cite{FC}, Lin \cite{Lin}, Wei and Xu \cite{WX} and the references therein.

In this paper, we will apply a variant (for nonlocal nonlinearity) of the direct method of moving spheres for fractional Laplacians due to Chen, Li and Zhang \cite{CLZ} to establish the following complete classification theorem for the Schr\"{o}dinger-Hartree equation \eqref{PDEH}.

\begin{thm}\label{Thm0}
Assume $n\geq2$, $n>2\alpha$, $0<\alpha\leq2$, $c_1,c_2\geq0$ with $c_{1}+c_{2}>0$, $0<p_{1}\leq1$ and $0<p_{2}\leq\frac{n+\alpha}{n-\alpha}$. Suppose $u$ is a nonnegative classical solution of \eqref{PDEH}. If $c_{1}(1-p_{1})+c_{2}(\frac{n+\alpha}{n-\alpha}-p_{2})=0$, then we have either $u\equiv0$ or $u$ must assume the following form
\begin{equation*}
  u(x)=C\left(\frac{\mu}{1+\mu^{2}|x-x_0|^2}\right)^{\frac{n-\alpha}{2}} \qquad \text{for some} \,\,\, \mu>0 \,\,\, \text{and} \,\,\, x_{0}\in\mathbb{R}^{n},
\end{equation*}
where the constant $C$ depends on $n,\alpha,c_{1},c_{2}$. If $c_{1}(1-p_{1})+c_{2}(\frac{n+\alpha}{n-\alpha}-p_{2})>0$, then $u\equiv0$ in $\mathbb{R}^{n}$.
\end{thm}

\begin{rem}\label{rem0}
Theorem \ref{Thm0} extends the classification results for \eqref{PDEH} in \cite{DFQ,Liu} from $c_{2}=0$ and $p_{1}=1$ to general cases $c_{2}\geq0$, $0<p_{1}\leq1$ and $0<p_{2}\leq\frac{n+\alpha}{n-\alpha}$.
\end{rem}

We will apply a variant (for nonlocal nonlinearity) of the direct method of moving spheres for fractional Laplacians developed by Chen, Li and Zhang \cite{CLZ} to prove Theorem \ref{Thm0}. More precisely, let us define the following notation
$$u_{x,\lambda}(y)=\left(\frac{\lambda}{|y-x|}\right)^{n-\alpha}u\left(\frac{\lambda^2(y-x)}{|y-x|^2}+x\right), \,\,\,\,\,\,\,\,\, \omega_{x,\lambda}(y)=u_{x,\lambda}(y)-u(y),$$
$$B^{-}_{\lambda}:=\{y\in B_{\lambda}(x)\setminus\{x\} \, | \, \omega_{x,\lambda}(y)<0\}.$$
The main ingredients in Chen, Li and Zhang's direct method of moving spheres for fractional Laplacians are maximum principles (i.e., \emph{Narrow region principle}) for the following problem
\begin{equation}\label{0-4}
  (-\Delta)^{\frac{\alpha}{2}}\omega_{x,\lambda}(y)+c(y)\omega_{x,\lambda}(y)\geq0 \,\,\,\,\, \text{in} \,\, \Omega \cap B_{\lambda}^-,
\end{equation}
where $\Omega\subseteq B_{\lambda}(x)\setminus\{x\}$ is a bounded region, $c(y)$ comes from applying mean value theorem to the difference between two nonlinearities defined point-wise and satisfies certain conditions. However, since the nonlinearities in our Schr\"{o}dinger-Hartree equation \eqref{PDEH} are nonlocal, the difference between two nonlinearities will become much more complicated and subtle. The differential inequality that one can derive from \eqref{PDEH} is the following
\begin{equation}\label{0-5}
  (-\Delta)^{\frac{\alpha}{2}}\omega_{x,\lambda}(y)-\mathcal{L}_{x,\lambda}(y)\omega_{x,\lambda}(y)-2c_{1}\left(\int_{\Sigma^{-}_{\lambda}}\frac{u(z)\omega_{x,\lambda}(z)}{|y-z|^{2\alpha}}dz\right)
  u^{p_{1}}(y)\geq0 \,\,\,\,\,\,\, \text{in} \,\,\,\, \Omega\cap B^{-}_{\lambda},
\end{equation}
from which we can observe that $\omega_{\lambda}$ will always appear in the convolution. It is difficult for us to simplify it into the desired form $c(y)\omega_{\lambda}(y)$. Fortunately, by more careful and refined analysis, we can circumvent this difficulty and establish the variants (for nonlocal nonlinearity) of the \emph{narrow region principle} for the problem \eqref{0-5} (see Theorem \ref{Thm2} in Section 2). We believe that the methods in this paper can be conveniently applied to study other fractional order equations with various nonlocal nonlinearities.

Through entirely similar arguments as in the proof of Theorem \ref{Thm0}, we will also classify all the nonnegative solutions to the following Schr\"{o}dinger-Maxwell equations with combined nonlinearities
\begin{equation}\label{PDEM}\\\begin{cases}
(-\Delta)^{\frac{\alpha}{2}}u(x)=c_{1}\Big(\frac{1}{|x|^{n-\alpha}}\ast|u|^{\frac{n+\alpha}{n-\alpha}}\Big)u^{q_{1}}(x)+c_2u^{q_{2}}(x), \,\,\,\,\,\,\,\, x\in\mathbb{R}^{n}, \\
u\in C^{1,1}_{loc}\cap\mathcal{L}_{\alpha}(\mathbb{R}^{n}), \,\,\,\,\, u(x)\geq0, \,\,\,\,\,\,\, x\in\mathbb{R}^{n},
\end{cases}\end{equation}
where $0<\alpha\leq2$, $n\geq2$, $n>\alpha$, $c_1,c_2\geq0$ with $c_{1}+c_{2}>0$, $0<q_{1}\leq\frac{2\alpha}{n-\alpha}$ and $0<q_{2}\leq\frac{n+\alpha}{n-\alpha}$. The Schr\"{o}dinger-Maxwell equations \eqref{PDEM} are equivalent to the following PDEs systems:
\begin{equation}\label{PDEM-S}\\\begin{cases}
(-\Delta)^{\frac{\alpha}{2}}u(x)=v(x)u^{q_{1}}(x)+c_2u^{q_{2}}(x), \,\,\,\,\,u(x)\geq0,\,\,\,\,\,\, x\in\mathbb{R}^{n}, \\
(-\Delta)^{\frac{\alpha}{2}}v(x)=c_{1}R^{-1}_{\alpha,n}\,u^{\frac{n+\alpha}{n-\alpha}}(x), \,\,\,\,\, v(x)\geq0, \,\,\,\,\,\,\, x\in\mathbb{R}^{n},
\end{cases}\end{equation}
where the Riesz potential's constants $R_{\alpha,n}:=\frac{\Gamma\big(\frac{n-\alpha}{2}\big)}{\pi^{\frac{n}{2}}2^{\alpha}\Gamma(\frac{\alpha}{2})}$ (see \cite{Stein}).

Chen and Li \cite{CL0} classified all the positive solutions to Schr\"{o}dinger-Maxwell equations \eqref{PDEM} with $c_{2}=0$ and $q_{1}=\frac{2\alpha}{n-\alpha}$ (see also \cite{XL}). In this paper, we will apply a variant (for nonlocal nonlinearity) of the direct method of moving spheres for fractional Laplacians due to Chen, Li and Zhang \cite{CLZ} to establish the following complete classification theorem for the Schr\"{o}dinger-Maxwell equation \eqref{PDEM}.

\begin{thm}\label{Thm1}
Assume $n\geq2$, $n>\alpha$, $0<\alpha\leq2$, $c_1,c_2\geq0$ with $c_{1}+c_{2}>0$, $0<q_{1}\leq\frac{2\alpha}{n-\alpha}$ and $0<q_{2}\leq\frac{n+\alpha}{n-\alpha}$. Suppose $u$ is a nonnegative classical solution of \eqref{PDEM}. If $c_{1}(\frac{2\alpha}{n-\alpha}-q_{1})+c_{2}(\frac{n+\alpha}{n-\alpha}-q_{2})=0$, then we have either $u\equiv0$ or $u$ must assume the following form
\begin{equation*}
  u(x)=C\left(\frac{\mu}{1+\mu^{2}|x-x_0|^2}\right)^{\frac{n-\alpha}{2}} \qquad \text{for some} \,\,\, \mu>0 \,\,\, \text{and} \,\,\, x_{0}\in\mathbb{R}^{n},
\end{equation*}
where the constant $C$ depends on $n,\alpha,c_{1},c_{2}$. If $c_{1}(\frac{2\alpha}{n-\alpha}-q_{1})+c_{2}(\frac{n+\alpha}{n-\alpha}-q_{2})>0$, then $u\equiv0$ in $\mathbb{R}^{n}$.
\end{thm}

Theorem \ref{Thm1} can be proved in a quite similar way as the proof of Theorem \ref{Thm0}, thus we will only mention some main ingredients in its proof in Section 3.

\begin{rem}\label{rem1}
Theorem \ref{Thm1} extends the classification results for \eqref{PDEM} in \cite{CL0,XL} from $c_{2}=0$ and $q_{1}=\frac{2\alpha}{n-\alpha}$ to general cases $c_{2}\geq0$, $0<q_{1}\leq\frac{2\alpha}{n-\alpha}$ and $0<q_{2}\leq\frac{n+\alpha}{n-\alpha}$.
\end{rem}

As a consequence of Theorem \ref{Thm1}, we have the following corollary on complete classification results for the Schr\"{o}dinger-Maxwell systems \eqref{PDEM-S}.
\begin{cor}\label{Cor1}
Assume $n\geq2$, $n>\alpha$, $0<\alpha\leq2$, $c_1,c_2\geq0$ with $c_{1}+c_{2}>0$, $0<q_{1}\leq\frac{2\alpha}{n-\alpha}$ and $0<q_{2}\leq\frac{n+\alpha}{n-\alpha}$. Suppose $(u,v)$ is a pair of nonnegative classical solutions of the system \eqref{PDEM-S}. Then, we have either $(u,v)\equiv(0,C_{0})$ for some $C_{0}\geq0$, or $(u,v)$ must assume the following forms
\begin{equation}\label{explicit}
  u(x)=C_{1}\left(\frac{\mu}{1+\mu^{2}|x-x_0|^2}\right)^{\frac{n-\alpha}{2}} \qquad \text{and} \qquad v(x)=C_{2}\left(\frac{\mu}{1+\mu^{2}|x-x_0|^2}\right)^{\frac{n-\alpha}{2}}
\end{equation}
for some $\mu>0$ and $x_{0}\in\mathbb{R}^{n}$, where the positive constants $C_{1}$ and $C_{2}$ depend on $n,\alpha,c_{1},c_{2}$. Moreover, if $(u,v)$ assume the form \eqref{explicit}, then we must have $c_{1}(\frac{2\alpha}{n-\alpha}-q_{1})+c_{2}(\frac{n+\alpha}{n-\alpha}-q_{2})=0$.
\end{cor}

The rest of our paper is organized as follows. In Section 2, we will carry out our proof of Theorem \ref{Thm0}. Section 3 and 4 are devoted to proving our Theorem \ref{Thm1} and Corollary \ref{Cor1} respectively.

In the following, we will use $C$ to denote a general positive constant that may depend on $n$, $\alpha$, $c_{1}$, $c_{2}$, $p_{1}$, $p_{2}$, $q_{1}$, $q_{2}$ and $u$, and whose value may differ from line to line.

\section{Proof of Theorem \ref{Thm0}}
In this section, we will use a direct method of moving spheres for nonlocal nonlinearity with the help of Narrow region principle to classify the nonnegative solutions of Schr\"{o}dinger-Hartree equation \eqref{PDEH}.

\subsection{The direct method of moving spheres for nonlocal nonlinearity}

Assume $n\geq2$, $n>2\alpha$, $0<\alpha\leq2$, $c_1,c_2\geq0$ with $c_{1}+c_{2}>0$, $0<p_{1}\leq1$ and $0<p_{2}\leq\frac{n+\alpha}{n-\alpha}$. Suppose $u$ is a nonnegative classical solution of \eqref{PDEH} which is not identically zero. It follows immediately that $u>0$ in $\mathbb{R}^{n}$ and $\int_{\mathbb{R}^{n}}\frac{u^{2}(x)}{|x|^{2\alpha}}dx<+\infty$. Thus we assume $u$ is actually a positive solution from now on.

For arbitrary $x\in\mathbb{R}^n$ and $\lambda>0$, we define the conformal transforms
$$u_{x,\lambda}(y):=\left(\frac{\lambda}{|y-x|}\right)^{n-\alpha}u(y^{x,\lambda}),\ \ \ \ \ \ \forall \,\, y\in\mathbb{R}^n\setminus\{x\},$$
where
$$y^{x,\lambda}=\frac{\lambda^2(y-x)}{|y-x|^2}+x.$$
Then, since $u$ is a positive classical solution of \eqref{PDEH}, one can verify that $u_{x,\lambda}\in\mathcal{L}_{\alpha}(\mathbb{R}^{n})\cap C^{1,1}_{loc}(\mathbb{R}^{n}\setminus\{x\})$ if $0<\alpha<2$ ($u_{x,\lambda}\in C^{2}(\mathbb{R}^{n}\setminus\{x\})$ if $\alpha=2$) and satisfies the integral property $$\int_{\mathbb{R}^{n}}\frac{u_{x,\lambda}^{2}(y)}{\lambda^{2\alpha}}dy=\int_{\mathbb{R}^{n}}\frac{u^{2}(x)}{|x|^{2\alpha}}dx<+\infty$$
and a similar equation as $u$ for any $x\in\mathbb{R}^n$ and $\lambda>0$. In fact, without loss of generality, we may assume $x=0$ for simplicity and get, for $0<\alpha<2$ ($\alpha=2$ is similar),
\begin{eqnarray}\label{2-1-1}
 \nonumber (-\Delta)^{\frac{\alpha}{2}}u_{0,\lambda}(y)&=&C_{\alpha,n}P.V.\int_{\mathbb{R}^n}\frac{\Big(\big(\frac{\lambda}{|y|}\big)^{n-\alpha}
-\big(\frac{\lambda}{|z|}\big)^{n-\alpha}\Big)u\big(\frac{\lambda^2y}{|y|^{2}}\big)+\big(\frac{\lambda}{|z|}\big)^{n-\alpha}
\Big(u\big(\frac{\lambda^2y}{|y|^{2}}\big)-
u\big(\frac{\lambda^2z}{|z|^{2}}\big)\Big)}{|y-z|^{n+\alpha}}dz \\
\nonumber &=& u\Big(\frac{\lambda^2y}{|y|^{2}}\Big)(-\Delta)^{\frac{\alpha}{2}}\left[\left(\frac{\lambda}{|y|}\right)^{n-\alpha}\right]
+C_{\alpha,n}P.V.\int_{\mathbb{R}^{n}}\frac{u\big(\frac{\lambda^2y}{|y|^{2}}\big)-u(z)}{\big|y-\frac{\lambda^2z}
{|z|^{2}}\big|^{n+\alpha}}\frac{\lambda^{n+\alpha}}{|z|^{n+\alpha}}dz\\
\nonumber &=&\frac{\lambda^{n+\alpha}}{|y|^{n+\alpha}}(-\Delta)^{\frac{\alpha}{2}}u\Big(\frac{\lambda^2y}{|y|^{2}}\Big)\\
\nonumber &=&c_1\frac{\lambda^{n+\alpha}}{|y|^{n+\alpha}}
\int_{\mathbb{R}^{n}}\frac{|u(z)|^{2}}{\big|\frac{\lambda^2y}{|y|^{2}}-z\big|^{2\alpha}}dz
  \cdot u^{p_{1}}\Big(\frac{\lambda^2y}{|y|^{2}}\Big)+c_2\frac{\lambda^{n+\alpha}}{|y|^{n+\alpha}}u^{p_{2}}\left(\frac{\lambda^2y}{|y|^2}\right) \\
 \nonumber &=&c_1\frac{\lambda^{n+\alpha}}{|y|^{n+\alpha}}\int_{\mathbb{R}^{n}}\frac{\lambda^{2n}|z|^{-2n}}{\big|\frac{\lambda^2y}
 {|y|^{2}}-\frac{\lambda^2z}{|z|^{2}}\big|^{2\alpha}}
 \Big|u\Big(\frac{\lambda^2z}{|z|^{2}}\Big)\Big|^{2}dz\cdot u^{p_{1}}\Big(\frac{\lambda^2y}{|y|^{2}}\Big)+ c_2\left(\frac{\lambda}{|y|}\right)^{\tau_{2}}u_{0,\lambda}^{p_{2}}(y)\\
 \nonumber &=& c_1\left(\frac{\lambda}{|y|}\right)^{\tau_{1}}\bigg[\frac{1}{|\cdot|^{2\alpha}}\ast|u_{0,\lambda}|^{2}\bigg](y)u^{p_{1}}_{0,\lambda}(y)+ c_2\left(\frac{\lambda}{|y|}\right)^{\tau_{2}}u_{0,\lambda}^{p_{2}}(y),
\end{eqnarray}
this means, the conformal transforms $u_{x,\lambda}\in\mathcal{L}_{\alpha}(\mathbb{R}^{n})\cap C^{1,1}_{loc}(\mathbb{R}^{n}\setminus\{x\})$ ($u_{x,\lambda}\in C^{2}(\mathbb{R}^{n}\setminus\{x\})$ if $\alpha=2$) satisfies
\begin{equation}\label{2-1-2}
   (-\Delta)^{\frac{\alpha}{2}}u_{x,\lambda}(y)=c_1\left(\frac{\lambda}{|y-x|}\right)^{\tau_{1}}\bigg(\frac{1}{|\cdot|^{2\alpha}}\ast u_{x,\lambda}^{2}\bigg)u^{p_{1}}_{x,\lambda}(y)+ c_2\left(\frac{\lambda}{|y-x|}\right)^{\tau_{2}}u_{x,\lambda}^{p_{2}}(y)
\end{equation}
for every $y\in\mathbb{R}^{n}\setminus\{x\}$, where $\tau_{1}:=(n-\alpha)(1-p_{1})\geq0$ and $\tau_{2}:=(n+\alpha)-p_{2}(n-\alpha)\geq0$. For any $\lambda>0$, we denote
$$B_\lambda(x):=\{y\in \mathbb{R}^n\,|\,|y-x|<\lambda\},$$
and define
$$P(y):=\bigg(\frac{1}{|\cdot|^{2\alpha}}\ast u^2\bigg)(y), \qquad \widetilde{P}_{x,\lambda}(y):=\int_{B_{\lambda}(x)}\frac{u(z)}{|y-z|^{2\alpha}}dz.$$

Let $\omega_{x,\lambda}(y)=u_{x,\lambda}(y)-u(y)$ for any $y\in B_{\lambda}(x)\setminus\{x\}$. By the definition of $u_{x,\lambda}$ and $\omega_{x,\lambda}$, we have
\begin{eqnarray}\label{2-0}
  \omega_{x,\lambda}(y)&=&u_{x,\lambda}(y)-u(y)=\left(\frac{\lambda}{|y-x|}\right)^{n-\alpha}u(y^{x,\lambda})-u(y) \\
 \nonumber &=&\left(\frac{\lambda}{|y-x|}\right)^{n-\alpha}\left(u(y^{x,\lambda})-\left(\frac{\lambda}{|y^{x,\lambda}-x|}\right)^{n-\alpha}u\big((y^{x,\lambda})^{x,\lambda}\big)\right) \\
 \nonumber &=&-\left(\frac{\lambda}{|y-x|}\right)^{n-\alpha}\omega_{x,\lambda}(y^{x,\lambda})=-\big(\omega_{x,\lambda}\big)_{x,\lambda}(y)
\end{eqnarray}
for every $y\in B_{\lambda}(x)\setminus\{x\}$.

We will first show that there exists a $\epsilon_{0}>0$ (depending on $x$) sufficiently small such that, for any $0<\lambda\leq\epsilon_{0}$, it holds that $\omega_{x,\lambda}(y)\geq 0$ for every $y\in B_\lambda(x)\setminus\{x\}$.

We first need to show that the nonnegative solution $u$ to \eqref{PDEH} also satisfies the following equivalent integral equation
\begin{equation}\label{IE}
  u(y)=\int_{\mathbb{R}^{n}}\frac{c_1R_{\alpha,n}}{|y-z|^{n-\alpha}}\bigg(\int_{\mathbb{R}^{n}}\frac{|u(\xi)|^{2}}{|z-\xi|^{2\alpha}}d\xi\bigg)u^{p_{1}}(z)dz
  +\int_{\mathbb{R}^{n}}\frac{c_2R_{\alpha,n}}{|y-z|^{n-\alpha}}u^{p_{2}}(z)dz,
\end{equation}
where the Riesz potential's constants $R_{\alpha,n}:=\frac{\Gamma\big(\frac{n-\alpha}{2}\big)}{\pi^{\frac{n}{2}}2^{\alpha}\Gamma(\frac{\alpha}{2})}$ (see \cite{Stein}).
\begin{lem}\label{equivalent}
Suppose $u$ is a nonnegative solution to \eqref{PDEH}, then $u$ also satisfies the equivalent integral equation \eqref{IE}, and vice versa.
\end{lem}
The proof of Lemma \ref{equivalent} is similar to \cite{CFY,DFQ,ZCCY}, so we omit the details here.

Based on Lemma \ref{equivalent}, we can prove that $\omega_{x,\lambda}$ has a strictly positive lower bound in a small neighborhood of $x$.
\begin{lem}\label{strictlypositive}
For each fixed $x\in\mathbb{R}^{n}$, there exists a $\eta_{0}>0$ (depending on $x$) sufficiently small such that, if $0<\lambda\leq\eta_{0}$, then
$$\omega_{x,\lambda}(y)\geq1,\ \ \   y\in \overline{B_{\lambda^{2}}(x)}\setminus\{x\}.$$
\end{lem}
\begin{proof}
We will prove Lemma \ref{strictlypositive} using the idea from \cite{CLZ}. Define $$f(u(y)):=c_1u^{p_{1}}(y)\int_{\mathbb{R}^{n}}\frac{u^{2}(z)}{|y-z|^{2\alpha}}dz+c_2u^{p_{2}}(y).$$
For any $|y|\geq1$, since $u>0$ also satisfy the integral equation \eqref{IE}, we can deduce that
\begin{equation}\label{a0}\begin{split}
u(y)&=R_{\alpha,n}\int_{\mathbb{R}^{n}}\frac{f(u(z))}{|y-z|^{n-\alpha}}dz\\
&\geq R_{\alpha,n}\int_{B_{\frac{1}{2}}(0)}\frac{f(u(z))}{|y-z|^{n-\alpha}}dz\\
&\geq \frac{C}{|y|^{n-\alpha}}\int_{B_{\frac{1}{2}}(0)}f(u(z))dz\\
&\geq \frac{C}{|y|^{n-\alpha}}.
\end{split}\end{equation}
It follows immediately that
\begin{equation}\label{a1}
  u_{x,\lambda}(y)=\left(\frac{\lambda}{|y-x|}\right)^{n-\alpha}u(y^{x,\lambda})\geq \left(\frac{\lambda}{|y-x|}\right)^{n-\alpha}\frac{C}{|y^{x,\lambda}|^{n-\alpha}}
=\frac{C}{\lambda^{n-\alpha}}
\end{equation}
for all $y\in\overline{B_{\lambda^{2}}(x)}\setminus\{x\}$. Therefore, we have if $0<\lambda\leq\eta_{0}$ for some $\eta_{0}(x)>0$ small enough, then
$$\omega_{x,\lambda}(y)=u_{x,\lambda}(y)-u(y)\geq \frac{C}{\lambda^{n-\alpha}}-\max_{|y-x|\leq\lambda^{2}}u(y)\geq1$$
for any $y\in \overline{B_{\lambda^{2}}(x)}\setminus\{x\}$, this finishes the proof of Lemma \ref{strictlypositive}.
\end{proof}

For every fixed $x\in \mathbb{R}^n$, define
\begin{equation*}
 B_\lambda^-=\{y\in B_\lambda(x)\setminus \{x\}\,|\,\omega_{x,\lambda}(y)<0\}.
\end{equation*}

Now we need the following theorem, which is a variant (for nonlocal nonlinearity) of the \emph{Narrow region principle} (Theorem 2.2 in \cite{CLZ}). \begin{thm}\label{Thm2}(Narrow region principle)
Assume $x\in\mathbb{R}^{n}$ is arbitrarily fixed. Let $\Omega$ be a narrow region in $B_{\lambda}(x)\setminus\{x\}$ with small thickness $0<l<\lambda$ such that $\Omega\subseteq A_{\lambda,l}(x):=\{y\in\mathbb{R}^n|\,\lambda-l<|y-x|<\lambda\}$. Suppose $\omega_{x,\lambda}\in\mathcal{L}_{\alpha}(\mathbb{R}^{n})\cap C^{1,1}_{loc}(\Omega)$ if $0<\alpha<2$ ($\omega_{x,\lambda}\in C^{2}(\Omega)$ if $\alpha=2$) and satisfies
\begin{equation}\label{nrp}\\\begin{cases}
(-\Delta)^{\frac{\alpha}{2}}\omega_{x,\lambda}(y)-\mathcal{L}_{x,\lambda}(y)\omega_{x,\lambda}(y)-2c_1\int_{B^{-}_{\lambda}}\frac{u(z)\omega_{x,\lambda}(z)}{|y-z|^{2\alpha}}dz \, u^{p_{1}}(y)\geq0 \,\,\,\,\, \text{in} \,\,\, \Omega\cap B^{-}_{\lambda},\\
\text{negative minimum of} \,\, \omega_{x,\lambda}\,\, \text{is attained in the interior of}\,\, B_{\lambda}(x)\setminus\{x\} \,\,\text{if} \,\,\, B^{-}_{\lambda}\neq\emptyset,\\
\text{negative minimum of} \,\,\, \omega_{x,\lambda} \,\,\, \text{cannot be attained in} \,\,\, (B_{\lambda}(x)\setminus\{x\})\setminus\Omega,
\end{cases}\end{equation}
where $\mathcal{L}_{x,\lambda}(y):=c_{1}p_{1}P(y)u_{x,\lambda}^{p_{1}-1}(y)+c_{2}p_{2}\max\left\{u^{p_{2}-1}(y),u_{x,\lambda}^{p_{2}-1}(y)\right\}$. Then, we have \\
(i) there exists a sufficiently small constant $\delta_{0}(x)>0$, such that, for all $0<\lambda\leq\delta_{0}$,
\begin{equation}\label{nrp1}
  \omega_{x,\lambda}(y)\geq0, \,\,\,\,\,\, \forall \,y\in\Omega;
\end{equation}
(ii) there exists a sufficiently small $l_{0}(x,\lambda)>0$ depending on $\lambda$ continuously, such that, for all $0<l\leq l_{0}$,
\begin{equation}\label{nrp2}
  \omega_{x,\lambda}(y)\geq0, \,\,\,\,\,\, \forall \,y\in\Omega.
\end{equation}
\end{thm}
\begin{proof}
Without loss of generality, we may assume $x=0$ here for simplicity. Suppose on contrary that \eqref{nrp1} and \eqref{nrp2} do not hold, we will obtain a contradiction for any $0<\lambda<\delta_{0}$ with constant $\delta_{0}$ small enough and any $0<l\leq l_{0}(\lambda)$ with $l_{0}(\lambda)$ sufficiently small respectively. By \eqref{nrp} and our hypothesis, there exists $\tilde{y}\in(\Omega\cap B^{-}_{\lambda})\subseteq A_{\lambda,l}(0):=\{y\in\mathbb{R}^{n}|\,\lambda-l<|y|<\lambda\}$ such that
\begin{equation}\label{nrp-1}
  \omega_{0,\lambda}(\tilde{y})=\min_{B_{\lambda}(0)\setminus\{0\}}\omega_{0,\lambda}(y)<0.
\end{equation}

We first consider the cases $0<\alpha<2$. Let $\tilde{\omega}_{0,\lambda}(y)=\omega_{0,\lambda}(y)-\omega_{0,\lambda}(\tilde{y})$, then $\tilde{\omega}_{0,\lambda}(\tilde{y})=0$ and
$$(-\Delta)^{\alpha/2}\tilde{\omega}_{0,\lambda}(y)=(-\Delta)^{\alpha/2}\omega_{0,\lambda}(y).$$
By the anti-symmetry property $\omega_{x,\lambda}(y)=-(\omega_{x,\lambda})_{x,\lambda}(y)$, it holds
\begin{equation*}\begin{split}
\left(\frac{\lambda}{|y|}\right)^{n-\alpha}\tilde{\omega}_{0,\lambda}(y^{0,\lambda})&=\left(\frac{\lambda}{|y|}\right)^{n-\alpha}\omega_{0,\lambda}(y^{0,\lambda})-
\left(\frac{\lambda}{|y|}\right)^{n-\alpha}\omega_{0,\lambda}(\tilde{y})\\
&=-\omega_{0,\lambda}(y)+\omega_{0,\lambda}(\tilde{y})-\left(1+\left(\frac{\lambda}{|y|}\right)^{n-\alpha}\right)\omega_{0,\lambda}(\tilde{y})\\
&=-\tilde{\omega}_{0,\lambda}(y)-\left(1+\left(\frac{\lambda}{|y|}\right)^{n-\alpha}\right)\omega_{0,\lambda}(\tilde{y}).
\end{split}\end{equation*}

As a consequence, it follows that
\begin{equation*}\begin{split}
(-\Delta)^{\alpha/2}\tilde{\omega}_{0,\lambda}(\tilde{y})&=C_{n,\alpha} \, P.V.\int_{\mathbb{R}^n}\frac{\tilde{\omega}_{0,\lambda}(\tilde{y})-\tilde{\omega}_{0,\lambda}(z)}{|\tilde{y}-z|^{n+\alpha}}dz\\
&=C_{n,\alpha} \, P.V.\int_{B_\lambda(0)}\frac{-\tilde{\omega}_{0,\lambda}(z)}{|\tilde{y}-z|^{n+\alpha}}dz+\int_{\mathbb{R}^n\setminus{B_\lambda(0)}}
\frac{-\tilde{\omega}_{0,\lambda}(z)}{|\tilde{y}-z|^{n+\alpha}}dz\\
&=C_{n,\alpha} \, P.V.\Bigg(\int_{B_\lambda(0)}\frac{-\tilde{\omega}_{0,\lambda}(z)}{|\tilde{y}-z|^{n+\alpha}}dz+\int_{\mathbb{R}^n\setminus{B_\lambda(0)}}
\frac{\left(\frac{\lambda}{|z|}\right)^{n-\alpha}\tilde{\omega}_{0,\lambda}(z^{0,\lambda})}{|\tilde{y}-z|^{n+\alpha}}dz\\
&\ \ \ +\int_{\mathbb{R}^n\setminus{B_\lambda(0)}}\frac{\Big(1+\big(\frac{\lambda}{|z|}\big)^{n-\alpha}\Big)\omega_{0,\lambda}(\tilde{y})}{|\tilde{y}-z|^{n+\alpha}}dz\Bigg)\\
&=C_{n,\alpha} \, P.V.\Bigg(\int_{B_\lambda(0)}\frac{-\tilde{\omega}_{0,\lambda}(z)}{|\tilde{y}-z|^{n+\alpha}}dz+\int_{B_\lambda(0)}\frac{
\tilde{\omega}_{0,\lambda}(z)}{\left|\frac{|z|\tilde{y}}{\lambda}-\frac{\lambda z}{|z|}\right|^{n+\alpha}}dz\\
&\ \ \ +\int_{\mathbb{R}^n\setminus{B_\lambda(0)}}\frac{\Big(1+\big(\frac{\lambda}{|z|}\big)^{n-\alpha}\Big)\omega_{0,\lambda}(\tilde{y})}{|\tilde{y}-z|^{n+\alpha}}dz\Bigg).\\
\end{split}\end{equation*}
Notice that, for any $z\in B_\lambda(0)\setminus\{0\}$,
$$\left|\frac{|z|\tilde{y}}{\lambda}-\frac{\lambda z}{|z|}\right|^2-|\tilde{y}-z|^2=\frac{(|\tilde{y}|^2-\lambda^2)(|z|^2-\lambda^2)}{\lambda^2}>0,$$
combining this with $\omega_{0,\lambda}(\tilde{y})<0$ gives that
\begin{equation}\begin{split}\label{small}
(-\Delta)^{\alpha/2}\omega_{0,\lambda}(\tilde{y})&\leq C_{n,\alpha}\omega_{0,\lambda}(\tilde{y})\int_{\mathbb{R}^n\setminus{B_\lambda(0)}}\frac{1}{|\tilde{y}-z|^{n+\alpha}}dz\\
&\leq C_{n,\alpha}\omega_{0,\lambda}(\tilde{y})\int_{(\mathbb{R}^n\setminus{B_\lambda(0)})\cap(B_{4l}(\tilde{y})\setminus{B_l(\tilde{y})})}\frac{1}{|\tilde{y}-z|^{n+\alpha}}dz\\
&\leq \frac{C}{l^\alpha}\omega_{0,\lambda}(\tilde{y})<0.
\end{split}\end{equation}

For $\alpha=2$, we can also obtain the same estimate as \eqref{small} at some point $y_{0}\in\Omega\cap B^{-}_{\lambda}$. To this end, we define
\begin{equation}\label{2-14}
  \phi(y):=\cos\frac{|y|-\lambda+l}{l},
\end{equation}
then it follows that $\phi(y)\in[\cos1,1]$ for any $y\in\overline{A_{\lambda,l}(0)}=\{y\in\mathbb{R}^{n}\,|\,\lambda-l\leq|y|\leq\lambda\}$ and $-\frac{\Delta\phi(y)}{\phi(y)}\geq\frac{1}{l^2}$. Define
\begin{equation}\label{2-15}
  \overline{\omega_{0,\lambda}}(y):=\frac{\omega_{0,\lambda}(y)}{\phi(y)}
\end{equation}
for $y\in\overline{A_{\lambda,l}(0)}$. Then there exists a $y_{0}\in\Omega\cap B^{-}_{\lambda}$ such that
\begin{equation}\label{2-16}
  \overline{\omega_{0,\lambda}}(y_{0})=\min_{\overline{A_{\lambda,l}(0)}}\overline{\omega_{0,\lambda}}(y)<0.
\end{equation}
Since
\begin{equation}\label{2-17}
  -\Delta\omega_{0,\lambda}(y_{0})=-\Delta\overline{\omega_{0,\lambda}}(y_{0})\phi(y_{0})-2\nabla\overline{\omega_{0,\lambda}}(y_{0})\cdot\nabla\phi(y_{0})
  -\overline{\omega_{0,\lambda}}(y_{0})\Delta\phi(y_{0}),
\end{equation}
one immediately has
\begin{equation}\label{2-18}
 -\Delta\omega_{0,\lambda}(y_{0})\leq\frac{1}{l^2}\omega_{0,\lambda}(y_{0}).
\end{equation}
In conclusion, we have proved that for both $0<\alpha<2$ and $\alpha=2$, there exists some $\hat{y}\in\Omega\cap B^{-}_{\lambda}$ such that
\begin{equation}\label{2-19}
  (-\Delta)^{\frac{\alpha}{2}}\omega_{0,\lambda}(\hat{y})\leq\frac{C}{l^{\alpha}}\omega_{0,\lambda}(\hat{y})<0.
\end{equation}

On the other hand, by \eqref{nrp}, we have at the point $\hat{y}$,
\begin{eqnarray}\label{nrp-4}
  0 &\leq& (-\Delta)^{\frac{\alpha}{2}}\omega_{0,\lambda}(\hat{y})-\mathcal{L}_{0,\lambda}(\hat{y})\omega_{0,\lambda}(\hat{y})
  -2c_1\int_{B^{-}_{\lambda}}\frac{u(z)\omega_{0,\lambda}(z)}{|\hat{y}-z|^{2\alpha}}dz \cdot u^{p_{1}}(\hat{y}) \\
 \nonumber &\leq& (-\Delta)^{\frac{\alpha}{2}}\omega_{0,\lambda}(\hat{y})-c_{0,\lambda}(\hat{y})\omega_{0,\lambda}(\hat{y}),
\end{eqnarray}
where
\begin{eqnarray*}
  && c_{x,\lambda}(y):=\mathcal{L}_{x,\lambda}(y)+2c_1\widetilde{P}_{x,\lambda}(y)u^{p_{1}}(y) \\
  &&\qquad\quad\,\, = c_{1}p_{1}P(y)u_{x,\lambda}^{p_{1}-1}(y)+c_{2}p_{2}\max\left\{u^{p_{2}-1}(y),u_{x,\lambda}^{p_{2}-1}(y)\right\}+2c_1\widetilde{P}_{x,\lambda}(y)u^{p_{1}}(y)>0.
\end{eqnarray*}

Since $\lambda-l<|y|<\lambda$, we have
\begin{eqnarray}\label{nrp-5}
  P(y) &\leq & \left\{\int_{|y-z|<\frac{|z|}{2}}+\int_{|y-z|\geq\frac{|z|}{2}}\right\}\frac{u^{2}(z)}{|y-z|^{2\alpha}}dz \\
 \nonumber &\leq& \Big[\max_{|y|\leq2\lambda}u(y)\Big]^{2}\int_{|y-z|<\lambda}\frac{1}{|y-z|^{2\alpha}}dz
 +4^{\alpha}\int_{\mathbb{R}^{n}}\frac{u^{2}(z)}{|z|^{2\alpha}}dz \\
 \nonumber &\leq& C\lambda^{n-2\alpha}\Big[\max_{|y|\leq2\lambda}u(y)\Big]^{2}+4^{\alpha}\int_{\mathbb{R}^{n}}\frac{u^{2}(x)}{|x|^{2\alpha}}dx=:C'_{\lambda},
\end{eqnarray}
and
\begin{eqnarray}\label{nrp-6}
\widetilde{P}_{0,\lambda}(y)&\leq&\int_{|y-z|<2\lambda}\frac{1}{|y-z|^{2\alpha}}u(z)dz\\
\nonumber &\leq&C\lambda^{n-2\alpha}\Big[\max_{|y|\leq4\lambda}u(y)\Big]=:C''_{\lambda}.
\end{eqnarray}
It is obvious that $C'_{\lambda}$ and $C''_{\lambda}$ depend on $\lambda$ continuously and monotone increasing with respect to $\lambda>0$.

Therefore, we infer from \eqref{a1}, \eqref{nrp-5} and \eqref{nrp-6} that, for any $\lambda-l\leq|y|\leq\lambda$ and $y\in B^{-}_{\lambda}$,
\begin{eqnarray}\label{nrp-8}
 && 0<c_{0,\lambda}(y)=c_{1}p_{1}P(y)u_{0,\lambda}^{p_{1}-1}(y)+c_{2}p_{2}\max\left\{u^{p_{2}-1}(y),u_{0,\lambda}^{p_{2}-1}(y)\right\}+2c_1\widetilde{P}_{0,\lambda}(y)u^{p_{1}}(y) \\
 \nonumber &&\,\,\,\, \leq c_{1}p_{1}C'_{\lambda}\left[\min_{|y|\leq\lambda}u_{0,\lambda}(y)\right]^{p_{1}-1}
 +c_{2}p_{2}\max\left\{\left(\max_{|y|\leq\lambda}u(y)\right)^{p_{2}-1},\left(\min_{|y|\leq\lambda}u_{0,\lambda}(y)\right)^{p_{2}-1}\right\} \\
 \nonumber && \quad\,\,\, +2c_1C''_{\lambda}\left[\max_{|y|\leq\lambda}u(y)\right]^{p_{1}}=:C_{\lambda},
\end{eqnarray}
where $C_{\lambda}$ depends continuously on $\lambda$ and monotone increasing with respect to $\lambda>0$.

As a consequence, it follows from \eqref{2-19}, \eqref{nrp-4} and \eqref{nrp-8} that
\begin{equation}\label{nrp-10}
0\leq(-\Delta)^{\frac{\alpha}{2}}\omega_{0,\lambda}(\hat{y})-c(\hat{y})\omega_{0,\lambda}(\hat{y})
\leq\left(\frac{C}{l^{\alpha}}-C_{\lambda}\right)\omega_{0,\lambda}(\hat{y}),
\end{equation}
that is,
\begin{equation}\label{nrp-12}
  \frac{C}{\lambda^{\alpha}}\leq\frac{C}{l^{\alpha}}\leq C_{\lambda}.
\end{equation}
We can derive a contradiction from \eqref{nrp-12} directly if $0<\lambda\leq\delta_{0}$ for some constant $\delta_{0}$ small enough, or if $0<l\leq l_{0}$ for some sufficiently small $l_{0}$ depending on $\lambda$ continuously. This implies that \eqref{nrp1} and \eqref{nrp2} must hold. Furthermore, by \eqref{nrp}, we can actually deduce from $\omega_{x,\lambda}(y)\geq0$ in $\Omega$ that
\begin{equation}\label{nrp-11}
  \omega_{x,\lambda}(y)\geq0, \,\,\,\,\,\, \forall \,\, y\in B_\lambda(x)\setminus\{x\}.
\end{equation}
This completes the proof of Theorem \ref{Thm2}.
\end{proof}

The following lemma provides a start point for us to move the spheres.

\begin{lem}\label{lemmasstart}
For every $x\in \mathbb{R}^n$, there exists $\epsilon_0(x)>0$ such that, $u_{x,\lambda}(y)\geq u(y)$ for all $\lambda\in(0,\epsilon_0(x)]$ and $y\in B_\lambda(x)\setminus \{x\}$.
\end{lem}
\begin{proof}
For every $x\in \mathbb{R}^n$, recall that
\begin{equation*}
 B_\lambda^-=\{y\in B_\lambda(x)\setminus \{x\}\,|\,\omega_{x,\lambda}(y)<0\}.
\end{equation*}

Take $\epsilon_{0}(x):=\min\{\eta_{0}(x),\delta_{0}(x)\}$, where $\eta_{0}(x)$ and $\delta_{0}(x)$ are defined the same as in Lemma \ref{strictlypositive} and Theorem \ref{Thm2}. We will show via contradiction arguments that, for any $0<\lambda\leq\epsilon_{0}$,
\begin{equation}\label{2-1-6}
B^{-}_{\lambda}=\emptyset.
\end{equation}

Suppose \eqref{2-1-6} does not hold, that is, $B^{-}_{\lambda}\neq\emptyset$ and hence $\omega_{x,\lambda}$ is negative somewhere in $B_{\lambda}(x)\setminus\{x\}$. For arbitrary $y\in B^{-}_{\lambda}$, we deduce from \eqref{PDEH} and \eqref{2-1-2} that
\begin{equation*}\begin{split}\label{2-1-7}
&\quad (-\Delta)^{\frac{\alpha}{2}}\omega_{x,\lambda}(y)\\
&\geq c_1\left(\left(\frac{1}{|\cdot|^{2\alpha}}\ast u_{x,\lambda}^{2}\right)(y)u^{p_{1}}_{x,\lambda}(y)
-\left(\frac{1}{|\cdot|^{2\alpha}}\ast u^{2}\right)(y)u^{p_{1}}(y)\right)
+c_2\big(u_{x,\lambda}^{p_{2}}(y)-u^{p_{2}}(y)\big)\\
&\geq c_1p_{1}\int_{\mathbb{R}^{n}}\frac{u^{2}(z)}{|y-z|^{2\alpha}}dz \, u_{x,\lambda}^{p_{1}-1}(y)\omega_{x,\lambda}(y)+c_2p_{2}\max\left\{u^{p_{2}-1}(y),u_{x,\lambda}^{p_{2}-1}(y)\right\}\omega_{x,\lambda}(y) \\
& \quad +c_1\int_{\mathbb{R}^{n}}\frac{u_{x,\lambda}^{2}(z)-u^2(z)}{|y-z|^{2\alpha}}dz\,u_{x,\lambda}^{p_{1}}(y) \\
&=\mathcal{L}_{x,\lambda}(y)\omega_{x,\lambda}(y)+c_1\int_{\mathbb{R}^{n}}\frac{u_{x,\lambda}^{2}(z)-u^2(z)}{|y-z|^{2\alpha}}dz\,u_{x,\lambda}^{p_{1}}(y) \\
&=\mathcal{L}_{x,\lambda}(y)\omega_{x,\lambda}(y)+c_1u_{x,\lambda}^{p_{1}}(y)\int_{B_\lambda(x)}\Bigg(\frac{1}{\left|\frac{(y-x)|z-x|}{\lambda}-\frac{\lambda (z-x)}{|z-x|}\right|^{2\alpha}}-\frac{1}{|y-z|^{2\alpha}}\Bigg)(u^2(z)-u_{x,\lambda}^{2}(z))dz \\
&\geq \mathcal{L}_{x,\lambda}(y)\omega_{x,\lambda}(y)+c_1u^{p_{1}}(y)\int_{B^-_\lambda(x)}\frac{1}{|y-z|^{2\alpha}}(u_{x,\lambda}^{2}(z)-u^2(z))dz \\
&\geq \mathcal{L}_{x,\lambda}(y)\omega_{x,\lambda}(y)+2c_1\left(\int_{B^{-}_{\lambda}}\frac{u(z)\omega_{x,\lambda}(z)}{|y-z|^{2\alpha}}dz\right)u^{p_{1}}(y),
\end{split}\end{equation*}
that is, for all $y\in B^{-}_{\lambda}$,
\begin{equation}\label{2-1-8}
  (-\Delta)^{\frac{\alpha}{2}}\omega_{x,\lambda}(y)-\mathcal{L}_{x,\lambda}(y)\omega_{x,\lambda}(y)
  -2c_1\left(\int_{B^{-}_{\lambda}}\frac{u(z)\omega_{x,\lambda}(z)}{|y-z|^{2\alpha}}dz\right)u^{p_{1}}(y)\geq0.
\end{equation}

Since $\epsilon_{0}(x):=\min\{\eta_{0}(x),\delta_{0}(x)\}$, by Lemma \ref{strictlypositive}, we have, for any $0<\lambda\leq\epsilon_{0}$,
\begin{equation}\label{4-1}
  \omega_{x,\lambda}(y)\geq1, \qquad \forall \,\, y\in \overline{B_{\lambda^{2}}(x)}\setminus\{x\}.
\end{equation}
Therefore, by taking $l=\lambda-\lambda^{2}$ and $\Omega=A_{\lambda,l}(x)$, then it follows from \eqref{2-1-8} and \eqref{4-1} that all the conditions in \eqref{nrp} in Theorem \ref{Thm2} are fulfilled, we can deduce from (i) in Theorem \ref{Thm2} that $\omega_{x,\lambda}\geq0$ in $\Omega=A_{\lambda,l}(x)$ for any $0<\lambda\leq\epsilon_{0}(x)$. That is, there exists $\epsilon_0(x)>0$ such that, for all $\lambda\in(0,\epsilon_0(x)]$,
$$\omega_{x,\lambda}(y)\geq 0, \qquad \forall \,\, y\in B_{\lambda}(x)\setminus\{x\}.$$
This completes the proof of Lemma \ref{lemmasstart}.
\end{proof}

For each fixed $x\in \mathbb{R}^n$, we define
\begin{equation}\label{defn}
  \bar{\lambda}(x)=\sup\{\lambda>0 \,|\, u_{x,\mu}\geq u\,\, \text{in} \,\, B_{\mu}(x)\setminus\{x\},\,\, \forall \,\, 0<\mu\leq\lambda\}.
\end{equation}
By Lemma \ref{lemmasstart}, $\bar{\lambda}(x)$ is well-defined and $0<\bar{\lambda}(x)\leq+\infty$ for any $x\in \mathbb{R}^n$.

We need the following Lemma, which is crucial in our proof.
\begin{lem}\label{lemmasequali}
If $\bar{\lambda}(\bar{x})<+\infty$ for some $\bar{x}\in \mathbb{R}^n$, then
\begin{equation*}\label{equality}
u_{\bar{x},\bar{\lambda}(\bar{x})}(y)=u(y),\,\,\,\,\,\,\,\, \forall \,\, y\in B_{\bar{\lambda}}(\bar{x})\setminus \{\bar{x}\}.
\end{equation*}
\end{lem}
\begin{proof}
Without loss of generality, we may assume $x=0$ for simplicity. Since $u$ is a positive solution to integral equation \eqref{IE}, one can verify that $u_{0,\lambda}$ also satisfies a similar integral equation as \eqref{IE} in $\mathbb{R}^{n}\setminus\{0\}$. In fact, by \eqref{IE} and direct calculations, we have, for any $y\in\mathbb{R}^{n}\setminus\{0\}$,
\begin{eqnarray}\label{2-1-41}
  \nonumber &&u_{0,\lambda}(y)=\left(\frac{\lambda}{|y|}\right)^{n-\alpha}u\left(\frac{\lambda^{2}y}{|y^{2}|}\right) \\
 \nonumber &=&\frac{\lambda^{n-\alpha}}{|y|^{n-\alpha}}\left(\int_{\mathbb{R}^{n}}\frac{c_1R_{\alpha,n}}{\big|\frac{\lambda^2y}{|y|^{2}}-z\big|^{n-\alpha}}
  \int_{\mathbb{R}^{n}}\frac{u^{2}(\xi)}{|z-\xi|^{2\alpha}}d\xi u^{p_{1}}(z)dz
  +\int_{\mathbb{R}^{n}}\frac{c_2R_{\alpha,n}}{\big|\frac{\lambda^2y}{|y|^{2}}-z\big|^{n-\alpha}}u^{p_{2}}(z)dz\right)\\
   \nonumber  &=& \frac{\lambda^{n-\alpha}}{|y|^{n-\alpha}}\int_{\mathbb{R}^{n}}\frac{c_1R_{\alpha,n}}{\big|\frac{\lambda^2y}{|y|^{2}}-\frac{\lambda^2z}{|z|^{2}}\big|^{n-\alpha}}
 \int_{\mathbb{R}^{n}}\frac{u^{2}(\frac{\lambda^2\xi}{|\xi|^{2}})}{\big|\frac{\lambda^2z}{|z|^{2}}
 -\frac{\lambda^2\xi}{|\xi|^{2}}\big|^{2\alpha}}\frac{\lambda^{2n}}{|\xi|^{2n}}d\xi
 u^{p_{1}}\left(\frac{\lambda^2z}{|z|^{2}}\right)\frac{\lambda^{2n}}{|z|^{2n}}dz\\
  \nonumber & \ \ \ +&\frac{\lambda^{n-\alpha}}{|y|^{n-\alpha}} \int_{\mathbb{R}^{n}}\frac{c_2R_{\alpha,n}}{\big|\frac{\lambda^2y}{|y|^{2}}-\frac{\lambda^2z}{|z|^{2}}\big|^{n-\alpha}}u^{p_{2}}\left(\frac{\lambda^2z}{|z|^{2}}\right)
  \frac{\lambda^{2n}}{|z|^{2n}}dz\\
  \nonumber  &=& \int_{\mathbb{R}^{n}}\frac{R_{\alpha,n}}{|y-z|^{n-\alpha}}\left[c_{1}\left(\frac{\lambda}{|z|}\right)^{\tau_{1}}
  \Big(\int_{\mathbb{R}^{n}}\frac{u_{0,\lambda}^{2}(\xi)}{|z-\xi|^{2\alpha}}d\xi\Big)u^{p_{1}}_{0,\lambda}(z)
  +c_{2}\left(\frac{\lambda}{|z|}\right)^{\tau_{2}}u_{0,\lambda}^{p_{2}}(z)\right]dz,
\end{eqnarray}
where $\tau_{1}:=(n-\alpha)(1-p_{1})\geq0$ and $\tau_{2}:=(n+\alpha)-p_{2}(n-\alpha)\geq0$.

Suppose on the contrary that $\omega_{0,\bar{\lambda}}\geq0$ but $\omega_{0,\bar{\lambda}}$ is not identically zero in $B_{\bar{\lambda}}(0)\setminus\{0\}$,
then we will get a contradiction with the definition \eqref{defn} of $\bar{\lambda}$. We first prove that
\begin{equation}\label{positive}
  \omega_{0,\bar{\lambda}}(y)>0, \,\,\,\,\,\, \forall \, y\in B_{\bar{\lambda}}(0)\setminus\{0\}.
\end{equation}
Indeed, if there exists a point $y^{0}\in B_{\bar{\lambda}}(0)\setminus\{0\}$ such that $\omega_{0,\bar{\lambda}}(y^{0})>0$, by continuity, there exists a small $\delta>0$ and a constant $c_{0}>0$ such that
\begin{equation*}\label{2-1-43}
B_{\delta}(y^{0})\subset B_{\bar{\lambda}}(0)\setminus\{0\} \,\,\,\,\,\, \text{and} \,\,\,\,\,\,
\omega_{0,\bar{\lambda}}(y)\geq c_{0}>0, \,\,\,\, \forall \, y\in B_{\delta}(y^{0}).
\end{equation*}
For any $y\in B_{\bar{\lambda}}(0)\setminus\{0\}$, one can derive that
\begin{equation*}\begin{split}
u(y)&=\int_{\mathbb{R}^{n}}\frac{c_1R_{\alpha,n}}{|y-z|^{n-\alpha}}P(z)u^{p_{1}}(z)dz
  +\int_{\mathbb{R}^{n}}\frac{c_2R_{\alpha,n}}{\big|y-z\big|^{n-\alpha}}u^{p_{2}}(z)dz\\
&=\int_{B_{\bar{\lambda}}(0)}\frac{c_1R_{\alpha,n}}{|y-z|^{n-\alpha}}P(z)u^{p_{1}}(z)dz
+\int_{B_{\bar{\lambda}}(0)}\frac{c_1R_{\alpha,n}}{|\frac{y|z|}{\bar{\lambda}}-\frac{\bar{\lambda}z}{|z|}|^{n-\alpha}}
P(z^{\bar{\lambda}})\left(\frac{\bar{\lambda}}{|z|}\right)^{2\alpha+\tau_{1}}u^{p_{1}}_{0,\bar{\lambda}}(z)dz\\
& +\int_{B_{\bar{\lambda}}(0)}\frac{c_2R_{\alpha,n}}{\big|y-z\big|^{n-\alpha}}u^{p_{2}}(z)dz
+\int_{B_{\bar{\lambda}}(0)}\frac{c_2R_{\alpha,n}}{\big|\frac{y|z|}{\bar{\lambda}}-\frac{\bar{\lambda}z}{|z|}\big|^{n-\alpha}}
\left(\frac{\bar{\lambda}}{|z|}\right)^{\tau_{2}}u_{0,\bar{\lambda}}^{p_{2}}(z)dz,\\
\end{split}\end{equation*}
and
\begin{equation*}\begin{split}
u_{0,\bar{\lambda}}(y)&=\int_{\mathbb{R}^{n}}\frac{c_1R_{\alpha,n}}{|y-z|^{n-\alpha}}\left(\frac{\bar{\lambda}}{|z|}\right)^{\tau_{1}}
\bar{P}_{0,\bar{\lambda}}(z)u^{p_{1}}_{0,\bar{\lambda}}(z)dz+\int_{\mathbb{R}^{n}}\frac{c_2R_{\alpha,n}}{\big|y-z\big|^{n-\alpha}}
\left(\frac{\bar{\lambda}}{|z|}\right)^{\tau_{2}}u_{0 ,\bar{\lambda}}^{p_{2}}(z)dz\\
&=\int_{B_{\bar{\lambda}}(0)}\frac{c_1R_{\alpha,n}}{|y-z|^{n-\alpha}}\left(\frac{\bar{\lambda}}{|z|}\right)^{\tau_{1}}
\bar{P}_{0,\bar{\lambda}}(z)u^{p_{1}}_{0,\bar{\lambda}}(z)dz \\
&\quad+\int_{B_{\bar{\lambda}}(0)}\frac{c_1R_{\alpha,n}}{|\frac{y|z|}{\bar{\lambda}}-\frac{\bar{\lambda}z}{|z|}|^{n-\alpha}}
\bar{P}_{0,\bar{\lambda}}(z^{\bar{\lambda}})\left(\frac{\bar{\lambda}}{|z|}\right)^{2\alpha}u^{p_{1}}(z)dz\\
&\quad+\int_{B_{\bar{\lambda}}(0)}\frac{c_2R_{\alpha,n}}{\big|y-z\big|^{n-\alpha}}
\left(\frac{\bar{\lambda}}{|z|}\right)^{\tau_{2}}u_{0,\bar{\lambda}}^{p_{2}}(z)dz
+\int_{B_{\bar{\lambda}}(0)}\frac{c_2R_{\alpha,n}}{\big|\frac{y|z|}
{\bar{\lambda}}-\frac{\bar{\lambda}z}{|z|}\big|^{n-\alpha}}u^{p_{2}}(z)dz,
\end{split}\end{equation*}
where
$$\bar{P}_{x,\lambda}(y):=\bigg(\frac{1}{|\cdot|^{2\alpha}}\ast u_{x,\lambda}^2\bigg)(y).$$

Let us define
$$K_{1,\bar{\lambda}}(y,z)=R_{\alpha,n}\left(\frac{1}{\big|y-z\big|^{n-\alpha}}
-\frac{1}{\big|\frac{y|z|}{\bar{\lambda}}-\frac{\bar{\lambda}z}{|z|}\big|^{n-\alpha}}\right),$$
$$K_{2,\bar{\lambda}}(y,z)=R_{\alpha,n}\left(\frac{1}{\big|y-z\big|^{2\alpha}}
-\frac{1}{\big|\frac{y|z|}{\bar{\lambda}}-\frac{\bar{\lambda}z}{|z|}\big|^{2\alpha}}\right).$$
It is easy to check that $K_{1,\bar{\lambda}}(y,z)>0$, $K_{2,\bar{\lambda}}(y,z)>0$, and
$$\bar{P}_{0,\bar{\lambda}}(z)=P(z^{\bar{\lambda}})\left(\frac{\bar{\lambda}}{|z|}\right)^{2\alpha},\ \ \quad \ \ \ P(z)=\bar{P}_{0,\bar{\lambda}}(z^{\bar{\lambda}})\left(\frac{\bar{\lambda}}{|z|}\right)^{2\alpha},$$
and furthermore,
$$\bar{P}_{0,\bar{\lambda}}(z)-P(z)=\int_{B_{\bar{\lambda}}(0)}K_{2,\bar{\lambda}}(z,\xi)\big(u^2_{0,\bar{\lambda}}(\xi)-u^2(\xi)\big)d\xi>0.$$
As a consequence, it follows immediately that, for any $y\in B_{\bar{\lambda}}(0)\setminus\{0\}$,
\begin{equation}\label{4-2}\begin{split}
\omega_{0,\bar{\lambda}}(y)&=c_1\int_{B_{\bar{\lambda}}(0)}K_{1,\bar{\lambda}}(y,z)P(z)
\left(\left(\frac{\bar{\lambda}}{|z|}\right)^{\tau_{1}}u^{p_{1}}_{0,\bar{\lambda}}(z)-u^{p_{1}}(z)\right)dz\\
&\ \ \ +c_1\int_{B_{\bar{\lambda}}(0)}K_{1,\bar{\lambda}}(y,z)(\bar{P}_{0,\bar{\lambda}}(z)-P(z))\left(\frac{\bar{\lambda}}{|z|}\right)^{\tau_{1}}u^{p_{1}}_{0,\bar{\lambda}}(z)dz\\
&\ \ \ +c_2\int_{B_{\bar{\lambda}}(0)}K_{1,\bar{\lambda}}(y,z)\left(\left(\frac{\bar{\lambda}}{|z|}\right)^{\tau_{2}}u^{p_{2}}_{0,\bar{\lambda}}(z)-u^{p_{2}}(z)\right)dz\\
&\geq c_1p_{1}\int_{B_{\delta}(y^0)}K_{1,\bar{\lambda}}(y,z)P(z)u^{p_{1}-1}_{0,\bar{\lambda}}(z)(u_{0,\bar{\lambda}}(z)-u(z))dz\\
&\quad+c_2p_{2}\int_{B_{\delta}(y^0)}K_{1,\bar{\lambda}}(y,z)\min\{u^{p_{2}-1}_{0,\bar{\lambda}}(z),u^{p_{2}-1}(z)\}\left(u_{0,\bar{\lambda}}(z)-u(z)\right)dz>0,
\end{split}\end{equation}
thus we arrive at \eqref{positive}. Furthermore, \eqref{4-2} also implies that there exists a $0<\eta<\bar{\lambda}$ small enough such that, for any $y\in \overline{B_{\eta}(0)}\setminus\{0\}$,
\begin{equation}\label{2-38}
  \omega_{0,\bar{\lambda}}(y)\geq c_1p_{1}\int_{B_{\frac{\delta}{2}}(y^{0})}c_{6}c_{5}c_{4}^{p_{1}-1}c_{0}\, dz
  +c_2p_{2}\int_{B_{\frac{\delta}{2}}(y^{0})}c_{6}c_{3}^{p_{2}-1}c_{0}\, dz=:\widetilde{c}_{0}>0.
\end{equation}

Now we define
\begin{equation}\label{2-39}
  \tilde{l}_{0}:=\min_{\lambda\in[\bar{\lambda},2\bar{\lambda}]}l_{0}(0,\lambda)>0,
\end{equation}
where $l_{0}(0,\lambda)$ is given by Theorem \ref{Thm2}. For a fixed small $0<r_{0}<\frac{1}{2}\min\{\tilde{l}_{0},\bar{\lambda}\}$, by \eqref{positive} and \eqref{2-38}, we can define
\begin{equation}\label{2-40}
  m_{0}:=\inf_{y\in\overline{B_{\bar{\lambda}-r_{0}}(0)}\setminus\{0\}}\omega_{0,\bar{\lambda}}(y)>0.
\end{equation}

Since $u$ is uniformly continuous on arbitrary compact set $K\subset\mathbb{R}^{n}$ (say, $K=\overline{B_{4\bar{\lambda}}(0)}$), we can deduce from \eqref{2-40} that, there exists a $0<\varepsilon_{0}<\frac{1}{2}\min\{\tilde{l}_{0},\bar{\lambda}\}$ sufficiently small, such that, for any $\lambda\in[\bar{\lambda},\bar{\lambda}+\varepsilon_{0}]$,
\begin{equation}\label{2-41}
  \omega_{0,\lambda}(y)\geq\frac{m_{0}}{2}>0, \,\,\,\,\,\, \forall \, y\in\overline{B_{\bar{\lambda}-r_{0}}(0)}\setminus\{0\}.
\end{equation}
In order to prove \eqref{2-41}, one should observe that \eqref{2-40} is equivalent to
\begin{equation}\label{2-42}
  |y|^{n-\alpha}u(y)-\bar{\lambda}^{n-\alpha}u(y^{0,\bar{\lambda}})\geq m_{0}\bar{\lambda}^{n-\alpha}, \,\,\,\,\,\,\,\,\, \forall \, |y|\geq\frac{\bar{\lambda}^{2}}{\bar{\lambda}-r_{0}}.
\end{equation}
Since $u$ is uniformly continuous on $\overline{B_{4\bar{\lambda}}(0)}$, we infer from \eqref{2-42} that there exists a $0<\varepsilon_{0}<\frac{1}{2}\min\{\tilde{l}_{0},\bar{\lambda}\}$ sufficiently small, such that, for any $\lambda\in[\bar{\lambda},\bar{\lambda}+\varepsilon_{0}]$,
\begin{equation}\label{2-43}
  |y|^{n-\alpha}u(y)-\lambda^{n-\alpha}u(y^{0,\lambda})\geq \frac{m_{0}}{2}\lambda^{n-\alpha}, \,\,\,\,\,\,\,\,\, \forall \, |y|\geq\frac{\lambda^{2}}{\bar{\lambda}-r_{0}},
\end{equation}
which is equivalent to \eqref{2-41}, hence we have proved \eqref{2-41}.

For any $\lambda\in[\bar{\lambda},\bar{\lambda}+\varepsilon_{0}]$, let $l:=\lambda-\bar{\lambda}+r_{0}\in(0,\tilde{l}_{0})$ and $\Omega:=A_{\lambda,l}(0)$, then it follows from \eqref{2-1-8} and \eqref{2-41} that all the conditions \eqref{nrp} in Theorem \ref{Thm2} are fulfilled, hence we can deduce from (ii) in Theorem \ref{Thm2} that
\begin{equation}\label{2-44}
  \omega_{0,\lambda}(y)\geq0, \quad\quad \forall \,\, y\in \Omega=A_{\lambda,l}(0).
\end{equation}
Therefore, we get from \eqref{2-41} and \eqref{2-44} that, $B^{-}_{\lambda}=\emptyset$ for all $\lambda\in[\bar{\lambda},\bar{\lambda}+\varepsilon_{0}]$, that is,
\begin{equation}\label{2-45}
  \omega_{0,\lambda}(y)\geq0, \,\,\,\,\,\,\, \forall \,\, y\in B_{\lambda}(0)\setminus\{0\},
\end{equation}
which contradicts with the definition \eqref{defn} of $\bar{\lambda}(0)$. As a consequence, in the case $0<\bar{\lambda}(0)<+\infty$, we must have $\omega_{0,\bar{\lambda}}\equiv0$ in $B_{\bar{\lambda}}(0)\setminus\{0\}$, that is,
\begin{equation}\label{2-46}
  u_{0,\bar{\lambda}(0)}(y)\equiv u(y), \,\,\,\,\,\, \,\,\, \forall \,\, y\in B_{\bar{\lambda}}(0)\setminus\{0\}.
\end{equation}
This finishes our proof of Lemma \ref{lemmasequali}.
\end{proof}

We also need the following property about the limiting radius $\bar{\lambda}(x)$.
\begin{lem}\label{lemma0}
If $\bar{\lambda}(\bar{x})=+\infty$ for some $\bar{x}\in \mathbb{R}^n$, then $\bar{\lambda}(x)=+\infty$ for all $x\in \mathbb{R}^n$.
\end{lem}
\begin{proof}
Since $\bar{\lambda}(\bar{x})=+\infty$, recalling the definition of $\bar{\lambda}$, we get
$$u_{\bar{x},\lambda}(y)\geq u(y),\ \ \ \forall \,\, y\in B_{\lambda}(\bar{x})\setminus\{\bar{x}\}, \quad\,\, \forall \,\, 0<\lambda<+\infty.$$
That is,
$$u(y)\geq u_{\bar{x},\lambda}(y),\ \quad \ \forall \,\,|y-\bar{x}|\geq\lambda, \,\,\,\,\, \forall \,\, 0<\lambda<+\infty.$$
It follows immediately that
\begin{equation}\label{infinite}
\lim_{|y|\rightarrow\infty}|y|^{n-\alpha}u(y)=+\infty.
\end{equation}
On the other hand, if we assume $\bar{\lambda}(x)<+\infty$ for some $x\in \mathbb{R}^n$, then by Lemma \ref{lemmasequali}, one arrives at
$$\lim_{|y|\rightarrow\infty}|y|^{n-\alpha}u(y)=\lim_{|y|\rightarrow\infty}|y|^{n-\alpha}u_{x,\bar{\lambda}(x)}(y)=(\bar{\lambda}(x))^{n-\alpha}u(x)<+\infty,$$
which contradicts with \eqref{infinite}. This finishes the proof of Lemma \ref{lemma0}.
\end{proof}

In the following two subsections, we will carry out the proof of Theorem \ref{Thm0} by discussing the critical cases and subcritical cases separately.

\subsection{Classification of positive solutions in the critical case $c_{1}(1-p_{1})+c_{2}(\frac{n+\alpha}{n-\alpha}-p_{2})=0$}
Without loss of generality, we may assume that $c_{1}>0$ and $c_{2}>0$, that is, $p_{1}=1$ and $p_{2}=\frac{n+\alpha}{n-\alpha}$. The Schr\"{o}dinger-Hartree equation \eqref{PDEH} is conformally invariant in such cases.

We carry out the proof by discussing two different possible cases.

\emph{Case (i).}  $\bar{\lambda}(x)=+\infty$ for all $x\in\mathbb{R}^{n}$. Therefore, for all $x\in\mathbb{R}^{n}$ and $0<\lambda<+\infty$, we have
$$u_{x,\lambda}(y)\geq u(y),\ \ \ \,\,\, \forall \,\, y\in B_{\lambda}(x)\setminus\{x\}, \,\,\,\,\,\,\, \forall \,\, 0<\lambda<+\infty.$$
By a calculus Lemma (Lemma 11.2 in \cite{LZH}), we must have $u\equiv C>0$, which contradicts with the equation \eqref{PDEH}.

\medskip

\emph{Case (ii).} By \emph{Case (i)} and Lemma \ref{lemma0}, we only need to consider the cases that
$$\bar{\lambda}(x)<\infty\ \ \text{for all}\ \ x\in \mathbb{R}^n.$$
From Lemma \ref{lemmasequali}, we infer that
\begin{equation}\label{4-3}
  u_{x,\bar{\lambda}(x)}(y)=u(y), \qquad \forall \,\, y\in B_{\bar{\lambda}(x)}(x)\setminus\{x\}.
\end{equation}
Since equation \eqref{PDEH} is conformally invariant, from a calculus lemma (Lemma 11.1 in \cite{LZH}) and \eqref{4-3}, we deduce that, there exists some $\mu>0$ and $x_{0}\in\mathbb{R}^{n}$ such that
$$u(x)=C\left(\frac{\mu}{1+\mu^{2}|x-x_0|^2}\right)^{\frac{n-\alpha}{2}},\ \ \quad \forall \,\, x\in\mathbb{R}^n,$$
where the constant $C$ depends on $n,\alpha,c_{1},c_{2}$.

\subsection{Nonexistence of positive solutions in the subcritical case $c_{1}(1-p_{1})+c_{2}(\frac{n+\alpha}{n-\alpha}-p_{2})>0$}
Without loss of generality, we may assume that $c_{1}(1-p_{1})>0$ and $c_{2}(\frac{n+\alpha}{n-\alpha}-p_{2})\geq0$, that is, $c_{1}>0$, $c_{2}\geq0$, $0<p_{1}<1$ and $0<p_{2}\leq\frac{n+\alpha}{n-\alpha}$. The Schr\"{o}dinger-Hartree equation \eqref{PDEH} involves at least one subcritical nonlinearities in such cases.

We will obtain a contradiction in both the following two different possible cases.

\emph{Case (i).}  $\bar{\lambda}(x)=+\infty$ for all $x\in\mathbb{R}^{n}$. Therefore, for all $x\in\mathbb{R}^{n}$ and $0<\lambda<+\infty$, we have
$$u_{x,\lambda}(y)\geq u(y),\ \ \ \,\,\, \forall \,\, y\in B_{\lambda}(x)\setminus\{x\}, \,\,\,\,\,\,\, \forall \,\, 0<\lambda<+\infty.$$
By a calculus Lemma (Lemma 11.2 in \cite{LZH}), we must have $u\equiv C>0$, which contradicts with the equation \eqref{PDEH}.

\medskip

\emph{Case (ii).} By \emph{Case (i)} and Lemma \ref{lemma0}, we only need to consider the cases that
$$\bar{\lambda}(x)<\infty\ \ \text{for all}\ \ x\in \mathbb{R}^n.$$
From Lemma \ref{lemmasequali}, we infer that
\begin{equation}\label{4-4}
  u_{x,\bar{\lambda}(x)}(y)=u(y), \qquad \forall \,\, y\in B_{\bar{\lambda}(x)}(x)\setminus\{x\}.
\end{equation}

Consider $x=0$, one can derive from \eqref{4-2} and \eqref{4-4} that
\begin{equation}\label{4-5}\begin{split}
0=\omega_{0,\bar{\lambda}}(y)&=c_1\int_{B_{\bar{\lambda}}(0)}K_{1,\bar{\lambda}}(y,z)P(z)
\left(\left(\frac{\bar{\lambda}}{|z|}\right)^{\tau_{1}}u^{p_{1}}_{0,\bar{\lambda}}(z)-u^{p_{1}}(z)\right)dz\\
&\ \ \ +c_1\int_{B_{\bar{\lambda}}(0)}K_{1,\bar{\lambda}}(y,z)(\bar{P}_{0,\bar{\lambda}}(z)-P(z))\left(\frac{\bar{\lambda}}{|z|}\right)^{\tau_{1}}u^{p_{1}}_{0,\bar{\lambda}}(z)dz\\
&\ \ \ +c_2\int_{B_{\bar{\lambda}}(0)}K_{1,\bar{\lambda}}(y,z)\left(\left(\frac{\bar{\lambda}}{|z|}\right)^{\tau_{2}}u^{p_{2}}_{0,\bar{\lambda}}(z)-u^{p_{2}}(z)\right)dz\\
&=c_1\int_{B_{\bar{\lambda}}(0)}K_{1,\bar{\lambda}}(y,z)P(z)
\left(\left(\frac{\bar{\lambda}}{|z|}\right)^{\tau_{1}}-1\right)u^{p_{1}}(z)dz \\
&\quad+c_2\int_{B_{\bar{\lambda}}(0)}K_{1,\bar{\lambda}}(y,z)\left(\left(\frac{\bar{\lambda}}{|z|}\right)^{\tau_{2}}-1\right)u^{p_{2}}(z)dz,
\end{split}\end{equation}
where
$$\bar{P}_{0,\bar{\lambda}}(z)-P(z)=\int_{B_{\bar{\lambda}}(0)}K_{2,\bar{\lambda}}(z,\xi)\big(u^2_{0,\bar{\lambda}}(\xi)-u^2(\xi)\big)d\xi=0,$$
and $\tau_{1}=(n-\alpha)(1-p_{1})>0$, $\tau_{2}=(n+\alpha)-p_{2}(n-\alpha)\geq0$.
As a consequence, it follows immediately that
$$0\geq c_1\int_{B_{\bar{\lambda}}(0)}K_{1,\bar{\lambda}}(y,z)P(z)
\left(\left(\frac{\bar{\lambda}}{|z|}\right)^{\tau_{1}}-1\right)u^{p_{1}}(z)dz>0,$$
which is absurd.

Thus we have ruled out both the \emph{Case (i)} and \emph{Case (ii)}, and hence \eqref{PDEH} does not admit any positive solutions. Therefore, the unique nonnegative solution to \eqref{PDEH} is $u\equiv0$.

This concludes our proof of Theorem \ref{Thm0}.

\section{Proof of Theorem \ref{Thm1}}

Theorem \ref{Thm1} can be proved in a quite similar way as the proof of Theorem \ref{Thm0}, thus we will only mention some main ingredients in its proof.

First, Suppose $u$ is a nonnegative classical solution of the Schr\"{o}dinger-Maxwell equation \eqref{PDEM} which is not identically zero. It follows immediately that $u>0$ in $\mathbb{R}^{n}$ and $\int_{\mathbb{R}^{n}}\frac{u^{\frac{n+\alpha}{n-\alpha}}(x)}{|x|^{n-\alpha}}dx<+\infty$. Then, one can verify that $u_{x,\lambda}\in\mathcal{L}_{\alpha}(\mathbb{R}^{n})\cap C^{1,1}_{loc}(\mathbb{R}^{n}\setminus\{x\})$ if $0<\alpha<2$ ($u_{x,\lambda}\in C^{2}(\mathbb{R}^{n}\setminus\{x\})$ if $\alpha=2$) satisfies the integral property $$\int_{\mathbb{R}^{n}}\frac{u_{x,\lambda}^{\frac{n+\alpha}{n-\alpha}}(y)}{\lambda^{n-\alpha}}dy
=\int_{\mathbb{R}^{n}}\frac{u^{\frac{n+\alpha}{n-\alpha}}(x)}{|x|^{n-\alpha}}dx<+\infty$$
and a similar equation as $u$ for any $x\in\mathbb{R}^n$ and $\lambda>0$. In fact, without loss of generality, we may assume $x=0$ for simplicity and get
\begin{eqnarray}\label{3-1-1}
 \nonumber &&(-\Delta)^{\frac{\alpha}{2}}u_{0,\lambda}(y)=\frac{\lambda^{n+\alpha}}{|y|^{n+\alpha}}(-\Delta)^{\frac{\alpha}{2}}u\Big(\frac{\lambda^2y}{|y|^{2}}\Big)\\
\nonumber &=&c_1\frac{\lambda^{n+\alpha}}{|y|^{n+\alpha}}
\int_{\mathbb{R}^{n}}\frac{|u(z)|^{\frac{n+\alpha}{n-\alpha}}}{\big|\frac{\lambda^2y}{|y|^{2}}-z\big|^{n-\alpha}}dz
  \cdot u^{q_{1}}\Big(\frac{\lambda^2y}{|y|^{2}}\Big)+c_2\frac{\lambda^{n+\alpha}}{|y|^{n+\alpha}}u^{q_{2}}\left(\frac{\lambda^2y}{|y|^2}\right) \\
 \nonumber &=&c_1\frac{\lambda^{n+\alpha}}{|y|^{n+\alpha}}\int_{\mathbb{R}^{n}}\frac{\lambda^{2n}|z|^{-2n}}{\big|\frac{\lambda^2y}
 {|y|^{2}}-\frac{\lambda^2z}{|z|^{2}}\big|^{n-\alpha}}
 \Big|u\Big(\frac{\lambda^2z}{|z|^{2}}\Big)\Big|^{\frac{n+\alpha}{n-\alpha}}dz\cdot u^{q_{1}}\Big(\frac{\lambda^2y}{|y|^{2}}\Big)+ c_2\left(\frac{\lambda}{|y|}\right)^{\sigma_{2}}u_{0,\lambda}^{q_{2}}(y)\\
 \nonumber &=& c_1\left(\frac{\lambda}{|y|}\right)^{\sigma_{1}}\bigg[\frac{1}{|\cdot|^{n-\alpha}}\ast|u_{0,\lambda}|^{\frac{n+\alpha}{n-\alpha}}\bigg](y)
 u^{q_{1}}_{0,\lambda}(y)+ c_2\left(\frac{\lambda}{|y|}\right)^{\sigma_{2}}u_{0,\lambda}^{q_{2}}(y),
\end{eqnarray}
this means, the conformal transforms $u_{x,\lambda}\in\mathcal{L}_{\alpha}(\mathbb{R}^{n})\cap C^{1,1}_{loc}(\mathbb{R}^{n}\setminus\{x\})$ ($u_{x,\lambda}\in C^{2}(\mathbb{R}^{n}\setminus\{x\})$ if $\alpha=2$) satisfies
\begin{equation}\label{3-1-2}
   (-\Delta)^{\frac{\alpha}{2}}u_{x,\lambda}(y)=c_1\left(\frac{\lambda}{|y-x|}\right)^{\sigma_{1}}\bigg(\frac{1}{|\cdot|^{n-\alpha}}\ast u_{x,\lambda}^{\frac{n+\alpha}{n-\alpha}}\bigg)u^{q_{1}}_{x,\lambda}(y)+ c_2\left(\frac{\lambda}{|y-x|}\right)^{\sigma_{2}}u_{x,\lambda}^{q_{2}}(y)
\end{equation}
for every $y\in\mathbb{R}^{n}\setminus\{x\}$, where $\sigma_{1}:=2\alpha-q_{1}(n-\alpha)\geq0$ and $\sigma_{2}:=(n+\alpha)-q_{2}(n-\alpha)\geq0$. Similar to Lemma \ref{equivalent}, we can also show that the nonnegative solution $u$ to \eqref{PDEM} also satisfies the following equivalent integral equation
\begin{equation}\label{IEM}
  u(y)=\int_{\mathbb{R}^{n}}\frac{c_1R_{\alpha,n}}{|y-z|^{n-\alpha}}\bigg(\int_{\mathbb{R}^{n}}\frac{|u(\xi)|^{\frac{n+\alpha}{n-\alpha}}}{|z-\xi|^{n-\alpha}}d\xi\bigg)u^{q_{1}}(z)dz
  +\int_{\mathbb{R}^{n}}\frac{c_2R_{\alpha,n}}{|y-z|^{n-\alpha}}u^{q_{2}}(z)dz,
\end{equation}
and vice versa.

Second, we define
$$Q(y):=\bigg(\frac{1}{|\cdot|^{n-\alpha}}\ast u^{\frac{n+\alpha}{n-\alpha}}\bigg)(y), \qquad \widetilde{Q}_{x,\lambda}(y):=\int_{B_{\lambda}(x)}\frac{u^{\frac{2\alpha}{n-\alpha}}(z)}{|y-z|^{n-\alpha}}dz.$$
We can prove the following \emph{Narrow region principle} through a quite similar way as the proof of Theorem \ref{Thm2} in Section 2.
\begin{thm}\label{Thm3}(Narrow region principle)
Assume $x\in\mathbb{R}^{n}$ is arbitrarily fixed. Let $\Omega$ be a narrow region in $B_{\lambda}(x)\setminus\{x\}$ with small thickness $0<l<\lambda$ such that $\Omega\subseteq A_{\lambda,l}(x):=\{y\in\mathbb{R}^n|\,\lambda-l<|y-x|<\lambda\}$. Suppose $\omega_{x,\lambda}\in\mathcal{L}_{\alpha}(\mathbb{R}^{n})\cap C^{1,1}_{loc}(\Omega)$ if $0<\alpha<2$ ($\omega_{x,\lambda}\in C^{2}(\Omega)$ if $\alpha=2$) and satisfies
\begin{equation}\label{nrp-3}\\\begin{cases}
(-\Delta)^{\frac{\alpha}{2}}\omega_{x,\lambda}(y)-\widetilde{\mathcal{L}}_{x,\lambda}(y)\omega_{x,\lambda}(y)
-c_1\frac{n+\alpha}{n-\alpha}\int_{B^{-}_{\lambda}}\frac{u^{\frac{2\alpha}{n-\alpha}}(z)\omega_{x,\lambda}(z)}{|y-z|^{n-\alpha}}dz \, u^{q_{1}}(y)\geq0 \,\,\,\,\, \text{in} \,\,\, \Omega\cap B^{-}_{\lambda},\\
\text{negative minimum of} \,\, \omega_{x,\lambda}\,\, \text{is attained in the interior of}\,\, B_{\lambda}(x)\setminus\{x\} \,\,\text{if} \,\,\, B^{-}_{\lambda}\neq\emptyset,\\
\text{negative minimum of} \,\,\, \omega_{x,\lambda} \,\,\, \text{cannot be attained in} \,\,\, (B_{\lambda}(x)\setminus\{x\})\setminus\Omega,
\end{cases}\end{equation}
where $\widetilde{\mathcal{L}}_{x,\lambda}(y):=c_{1}q_{1}Q(y)\max\left\{u^{q_{1}-1}(y),u_{x,\lambda}^{q_{1}-1}(y)\right\}
+c_{2}q_{2}\max\left\{u^{q_{2}-1}(y),u_{x,\lambda}^{q_{2}-1}(y)\right\}$. Then, we have \\
(i) there exists a sufficiently small constant $\delta_{0}(x)>0$, such that, for all $0<\lambda\leq\delta_{0}$,
\begin{equation}\label{nrp1-3}
  \omega_{x,\lambda}(y)\geq0, \,\,\,\,\,\, \forall \,y\in\Omega;
\end{equation}
(ii) there exists a sufficiently small $l_{0}(x,\lambda)>0$ depending on $\lambda$ continuously, such that, for all $0<l\leq l_{0}$,
\begin{equation}\label{nrp2-3}
  \omega_{x,\lambda}(y)\geq0, \,\,\,\,\,\, \forall \,y\in\Omega.
\end{equation}
\end{thm}
\begin{proof}
Without loss of generality, we may assume $x=0$ here for simplicity. Theorem \ref{Thm3} can be proved in a quite similar way as the proof of Theorem \ref{Thm2}, thus we will only mention the following key estimates for $Q(y)$ and $\widetilde{Q}_{0,\lambda}(y)$ for any $y\in A_{\lambda,l}(0)$.

Indeed, since $\lambda-l<|y|<\lambda$, we have
\begin{eqnarray}\label{nrp-5-3}
  Q(y) &\leq & \left\{\int_{|y-z|<\frac{|z|}{2}}+\int_{|y-z|\geq\frac{|z|}{2}}\right\}\frac{u^{\frac{n+\alpha}{n-\alpha}}(z)}{|y-z|^{n-\alpha}}dz \\
 \nonumber &\leq& \Big[\max_{|y|\leq2\lambda}u(y)\Big]^{\frac{n+\alpha}{n-\alpha}}\int_{|y-z|<\lambda}\frac{1}{|y-z|^{n-\alpha}}dz
 +2^{n-\alpha}\int_{\mathbb{R}^{n}}\frac{u^{\frac{n+\alpha}{n-\alpha}}(z)}{|z|^{n-\alpha}}dz \\
 \nonumber &\leq& C\lambda^{\alpha}\Big[\max_{|y|\leq2\lambda}u(y)\Big]^{\frac{n+\alpha}{n-\alpha}}
 +2^{n-\alpha}\int_{\mathbb{R}^{n}}\frac{u^{\frac{n+\alpha}{n-\alpha}}(x)}{|x|^{n-\alpha}}dx=:\widetilde{C'_{\lambda}},
\end{eqnarray}
and
\begin{eqnarray}\label{nrp-6-3}
\widetilde{Q}_{0,\lambda}(y)&\leq&\int_{|y-z|<2\lambda}\frac{1}{|y-z|^{n-\alpha}}u^{\frac{2\alpha}{n-\alpha}}(z)dz\\
\nonumber &\leq&C\lambda^{\alpha}\Big[\max_{|y|\leq4\lambda}u(y)\Big]^{\frac{2\alpha}{n-\alpha}}=:\widetilde{C''_{\lambda}}.
\end{eqnarray}
It is obvious that $\widetilde{C'_{\lambda}}$ and $\widetilde{C''_{\lambda}}$ depend on $\lambda$ continuously and monotone increasing with respect to $\lambda>0$.

The rest of the proof is completely similar to the proof of Theorem \ref{Thm2}, so we omit the details. This finishes our proof of Theorem \ref{Thm3}.
\end{proof}

Third, for each fixed $x\in \mathbb{R}^n$, we define the limiting radius by
\begin{equation}\label{defn-3}
  \bar{\lambda}(x)=\sup\{\lambda>0 \,|\, u_{x,\mu}\geq u\,\, \text{in} \,\, B_{\mu}(x)\setminus\{x\},\,\, \forall \,\, 0<\mu\leq\lambda\}\in(0,+\infty].
\end{equation}
Then, similar to Lemma \ref{lemmasequali} in Section 2, we also need the following Lemma, which is crucial in our proof.
\begin{lem}\label{lemmasequali-3}
If $\bar{\lambda}(\bar{x})<+\infty$ for some $\bar{x}\in \mathbb{R}^n$, then
\begin{equation*}\label{equality}
u_{\bar{x},\bar{\lambda}(\bar{x})}(y)=u(y),\,\,\,\,\,\,\,\, \forall \,\, y\in B_{\bar{\lambda}}(\bar{x})\setminus \{\bar{x}\}.
\end{equation*}
\end{lem}
\begin{proof}
Without loss of generality, we may assume $x=0$ for simplicity. Since $u$ is a positive solution to integral equation \eqref{IEM}, one can verify that $u_{0,\lambda}$ also satisfies a similar integral equation as \eqref{IEM} in $\mathbb{R}^{n}\setminus\{0\}$. In fact, by \eqref{IEM} and direct calculations, we have, for any $y\in\mathbb{R}^{n}\setminus\{0\}$,
\begin{eqnarray}\label{3-1-41}
   &&u_{0,\lambda}(y)=\left(\frac{\lambda}{|y|}\right)^{n-\alpha}u\left(\frac{\lambda^{2}y}{|y^{2}|}\right) \\
  \nonumber  &=& \int_{\mathbb{R}^{n}}\frac{R_{\alpha,n}}{|y-z|^{n-\alpha}}\left[c_{1}\left(\frac{\lambda}{|z|}\right)^{\sigma_{1}}
  \left(\int_{\mathbb{R}^{n}}\frac{u_{0,\lambda}^{\frac{n+\alpha}{n-\alpha}}(\xi)}{|z-\xi|^{n-\alpha}}d\xi\right)u^{q_{1}}_{0,\lambda}(z)
  +c_{2}\left(\frac{\lambda}{|z|}\right)^{\sigma_{2}}u_{0,\lambda}^{q_{2}}(z)\right]dz,
\end{eqnarray}
where $\sigma_{1}:=2\alpha-q_{1}(n-\alpha)\geq0$ and $\sigma_{2}:=(n+\alpha)-q_{2}(n-\alpha)\geq0$.

Suppose on the contrary that $\omega_{0,\bar{\lambda}}\geq0$ but $\omega_{0,\bar{\lambda}}$ is not identically zero in $B_{\bar{\lambda}}(0)\setminus\{0\}$,
then we will get a contradiction with the definition \eqref{defn-3} of $\bar{\lambda}$.

We first prove that
\begin{equation}\label{positive-3}
  \omega_{0,\bar{\lambda}}(y)>0, \,\,\,\,\,\, \forall \, y\in B_{\bar{\lambda}}(0)\setminus\{0\}.
\end{equation}
Indeed, if there exists a point $y^{0}\in B_{\bar{\lambda}}(0)\setminus\{0\}$ such that $\omega_{0,\bar{\lambda}}(y^{0})>0$, by continuity, there exists a small $\delta>0$ and a constant $c_{0}>0$ such that
\begin{equation*}\label{3-1-43}
B_{\delta}(y^{0})\subset B_{\bar{\lambda}}(0)\setminus\{0\} \,\,\,\,\,\, \text{and} \,\,\,\,\,\,
\omega_{0,\bar{\lambda}}(y)\geq c_{0}>0, \,\,\,\, \forall \, y\in B_{\delta}(y^{0}).
\end{equation*}
For any $y\in B_{\bar{\lambda}}(0)\setminus\{0\}$, by \eqref{IEM}, \eqref{3-1-41} and direct calculations, one can derive that
\begin{equation*}\begin{split}
u(y)&=\int_{\mathbb{R}^{n}}\frac{c_1R_{\alpha,n}}{|y-z|^{n-\alpha}}Q(z)u^{q_{1}}(z)dz
  +\int_{\mathbb{R}^{n}}\frac{c_2R_{\alpha,n}}{\big|y-z\big|^{n-\alpha}}u^{q_{2}}(z)dz\\
&=\int_{B_{\bar{\lambda}}(0)}\frac{c_1R_{\alpha,n}}{|y-z|^{n-\alpha}}Q(z)u^{q_{1}}(z)dz
+\int_{B_{\bar{\lambda}}(0)}\frac{c_1R_{\alpha,n}}{|\frac{y|z|}{\bar{\lambda}}-\frac{\bar{\lambda}z}{|z|}|^{n-\alpha}}
Q(z^{\bar{\lambda}})\left(\frac{\bar{\lambda}}{|z|}\right)^{n-\alpha+\sigma_{1}}u^{q_{1}}_{0,\bar{\lambda}}(z)dz\\
& +\int_{B_{\bar{\lambda}}(0)}\frac{c_2R_{\alpha,n}}{\big|y-z\big|^{n-\alpha}}u^{q_{2}}(z)dz
+\int_{B_{\bar{\lambda}}(0)}\frac{c_2R_{\alpha,n}}{\big|\frac{y|z|}{\bar{\lambda}}-\frac{\bar{\lambda}z}{|z|}\big|^{n-\alpha}}
\left(\frac{\bar{\lambda}}{|z|}\right)^{\sigma_{2}}u_{0,\bar{\lambda}}^{q_{2}}(z)dz,\\
\end{split}\end{equation*}
and
\begin{equation*}\begin{split}
u_{0,\bar{\lambda}}(y)&=\int_{\mathbb{R}^{n}}\frac{c_1R_{\alpha,n}}{|y-z|^{n-\alpha}}\left(\frac{\bar{\lambda}}{|z|}\right)^{\sigma_{1}}
\bar{Q}_{0,\bar{\lambda}}(z)u^{q_{1}}_{0,\bar{\lambda}}(z)dz+\int_{\mathbb{R}^{n}}\frac{c_2R_{\alpha,n}}{\big|y-z\big|^{n-\alpha}}
\left(\frac{\bar{\lambda}}{|z|}\right)^{\sigma_{2}}u_{0 ,\bar{\lambda}}^{q_{2}}(z)dz\\
&=\int_{B_{\bar{\lambda}}(0)}\frac{c_1R_{\alpha,n}}{|y-z|^{n-\alpha}}\left(\frac{\bar{\lambda}}{|z|}\right)^{\sigma_{1}}
\bar{Q}_{0,\bar{\lambda}}(z)u^{q_{1}}_{0,\bar{\lambda}}(z)dz \\
&\quad+\int_{B_{\bar{\lambda}}(0)}\frac{c_1R_{\alpha,n}}{|\frac{y|z|}{\bar{\lambda}}-\frac{\bar{\lambda}z}{|z|}|^{n-\alpha}}
\bar{Q}_{0,\bar{\lambda}}(z^{\bar{\lambda}})\left(\frac{\bar{\lambda}}{|z|}\right)^{n-\alpha}u^{q_{1}}(z)dz\\
&\quad+\int_{B_{\bar{\lambda}}(0)}\frac{c_2R_{\alpha,n}}{\big|y-z\big|^{n-\alpha}}
\left(\frac{\bar{\lambda}}{|z|}\right)^{\sigma_{2}}u_{0,\bar{\lambda}}^{q_{2}}(z)dz
+\int_{B_{\bar{\lambda}}(0)}\frac{c_2R_{\alpha,n}}{\big|\frac{y|z|}
{\bar{\lambda}}-\frac{\bar{\lambda}z}{|z|}\big|^{n-\alpha}}u^{q_{2}}(z)dz,
\end{split}\end{equation*}
where
$$\bar{Q}_{x,\lambda}(y):=\bigg(\frac{1}{|\cdot|^{n-\alpha}}\ast u_{x,\lambda}^{\frac{n+\alpha}{n-\alpha}}\bigg)(y).$$
Let us recall that
$$K_{1,\bar{\lambda}}(y,z):=R_{\alpha,n}\left(\frac{1}{\big|y-z\big|^{n-\alpha}}
-\frac{1}{\big|\frac{y|z|}{\bar{\lambda}}-\frac{\bar{\lambda}z}{|z|}\big|^{n-\alpha}}\right).$$
It is easy to check that $K_{1,\bar{\lambda}}(y,z)>0$, and
$$\bar{Q}_{0,\bar{\lambda}}(z)=Q(z^{\bar{\lambda}})\left(\frac{\bar{\lambda}}{|z|}\right)^{n-\alpha},\ \ \quad \ \ \ Q(z)=\bar{Q}_{0,\bar{\lambda}}(z^{\bar{\lambda}})\left(\frac{\bar{\lambda}}{|z|}\right)^{n-\alpha},$$
and furthermore,
$$\bar{Q}_{0,\bar{\lambda}}(z)-Q(z)=\int_{B_{\bar{\lambda}}(0)}K_{1,\bar{\lambda}}(z,\xi)
\big(u^{\frac{n+\alpha}{n-\alpha}}_{0,\bar{\lambda}}(\xi)-u^{\frac{n+\alpha}{n-\alpha}}(\xi)\big)d\xi>0.$$
As a consequence, it follows immediately that, for any $y\in B_{\bar{\lambda}}(0)\setminus\{0\}$,
\begin{equation}\label{4-2-3}\begin{split}
\omega_{0,\bar{\lambda}}(y)&=c_1\int_{B_{\bar{\lambda}}(0)}K_{1,\bar{\lambda}}(y,z)Q(z)
\left(\left(\frac{\bar{\lambda}}{|z|}\right)^{\sigma_{1}}u^{q_{1}}_{0,\bar{\lambda}}(z)-u^{q_{1}}(z)\right)dz\\
&\ \ \ +c_1\int_{B_{\bar{\lambda}}(0)}K_{1,\bar{\lambda}}(y,z)(\bar{Q}_{0,\bar{\lambda}}(z)-Q(z))\left(\frac{\bar{\lambda}}{|z|}\right)^{\sigma_{1}}u^{q_{1}}_{0,\bar{\lambda}}(z)dz\\
&\ \ \ +c_2\int_{B_{\bar{\lambda}}(0)}K_{1,\bar{\lambda}}(y,z)\left(\left(\frac{\bar{\lambda}}{|z|}\right)^{\sigma_{2}}u^{q_{2}}_{0,\bar{\lambda}}(z)-u^{q_{2}}(z)\right)dz\\
&\geq c_1q_{1}\int_{B_{\delta}(y^0)}K_{1,\bar{\lambda}}(y,z)Q(z)\min\{u^{q_{1}-1}_{0,\bar{\lambda}}(z),u^{q_{1}-1}(z)\}\left(u_{0,\bar{\lambda}}(z)-u(z)\right)dz\\
&\quad+c_2q_{2}\int_{B_{\delta}(y^0)}K_{1,\bar{\lambda}}(y,z)\min\{u^{q_{2}-1}_{0,\bar{\lambda}}(z),u^{q_{2}-1}(z)\}\left(u_{0,\bar{\lambda}}(z)-u(z)\right)dz>0,
\end{split}\end{equation}
thus we arrive at \eqref{positive-3}. Furthermore, \eqref{4-2-3} also implies that there exists a $0<\hat{\eta}<\bar{\lambda}$ small enough such that, for any $y\in \overline{B_{\hat{\eta}}(0)}\setminus\{0\}$,
\begin{equation}\label{3-38}
  \omega_{0,\bar{\lambda}}(y)\geq c_1q_{1}\int_{B_{\frac{\delta}{2}}(y^{0})}c_{6}c_{5}c_{4}^{q_{1}-1}c_{0}\, dz
  +c_2q_{2}\int_{B_{\frac{\delta}{2}}(y^{0})}c_{6}c_{3}^{q_{2}-1}c_{0}\, dz=:\hat{c}_{0}>0.
\end{equation}

The rest of the proof is entirely similar to the proof of Lemma \ref{lemmasequali}, by using \eqref{positive-3} and \eqref{3-38}, we can show that there exists a $\varepsilon_{1}>0$ small enough such that, for all $\lambda\in[\bar{\lambda},\bar{\lambda}+\varepsilon_{1}]$,
\begin{equation}\label{3-45}
  \omega_{0,\lambda}(y)\geq0, \,\,\,\,\,\,\, \forall \,\, y\in B_{\lambda}(0)\setminus\{0\},
\end{equation}
which contradicts with the definition \eqref{defn-3} of $\bar{\lambda}(0)$, so we omit the details. This completes our proof of Lemma \ref{lemmasequali-3}.
\end{proof}

The rest of the proof is completely similar to the proof of Theorem \ref{Thm0}, so we omit the details. This concludes our proof of Theorem \ref{Thm1}.

\section{Proof of Corollary \ref{Cor1}}

In this section, we will give a brief proof of Corollary \ref{Cor1} by using Theorem \ref{Thm1}.

First, we can prove that the nonnegative solution $(u,v)$ to PDEs system \eqref{PDEM-S} also satisfies the following equivalent IEs system
\begin{equation}\label{IEM-S}\\\begin{cases}
u(x)=\int_{\mathbb{R}^{n}}\frac{R_{\alpha,n}}{|x-y|^{n-\alpha}}\left(v(y)u^{q_{1}}(y)+c_2u^{q_{2}}(y)\right)dy, \,\,\,\,\,\,\,\,\,\, x\in\mathbb{R}^{n}, \\
v(x)=\int_{\mathbb{R}^{n}}\frac{c_{1}}{|x-y|^{n-\alpha}}u^{\frac{n+\alpha}{n-\alpha}}(y)dy+C_{0}, \,\,\,\, \,\,\,\,\,\,\, x\in\mathbb{R}^{n},
\end{cases}\end{equation}
where $C_{0}\geq0$ is a nonnegative real number. The proof is similar to \cite{CFY,DFQ,ZCCY}, so we omit the details here.

Therefore, $u$ satisfies the following equation
\begin{equation}\label{5-1}
  (-\Delta)^{\frac{\alpha}{2}}u(x)=c_{1}\Big(\frac{1}{|x|^{n-\alpha}}\ast|u|^{\frac{n+\alpha}{n-\alpha}}\Big)u^{q_{1}}(x)+C_{0}u^{q_{1}}(x)+c_2u^{q_{2}}(x), \,\,\,\,\,\,\,\, x\in\mathbb{R}^{n}.
\end{equation}

In the following, we will discuss two different possible cases respectively.

\emph{Case (i)} $C_{0}>0$. In such cases, noting that $0<q_{1}\leq\frac{2\alpha}{n-\alpha}<\frac{n+\alpha}{n-\alpha}$, the Schr\"{o}dinger-Maxwell equation \eqref{5-1} involves at least one subcritical nonlinear term $C_{0}u^{q_{1}}(x)$, thus it is clear from the proof of Theorem \ref{Thm1} (see Section 3) that,
$u\equiv0$ in $\mathbb{R}^{n}$, and hence $(u,v)\equiv(0,C_{0})$.

\emph{Case (ii)} $C_{0}=0$. We will discuss two different possible sub-cases separately.

\emph{Sub-case (i).} If $c_{1}(\frac{2\alpha}{n-\alpha}-q_{1})+c_{2}(\frac{n+\alpha}{n-\alpha}-q_{2})=0$, by Theorem \ref{Thm1} and \eqref{5-1}, we have either $u\equiv0$ or $u$ must assume the following form
\begin{equation}\label{5-2}
  u(x)=C_{1}\left(\frac{\mu}{1+\mu^{2}|x-x_0|^2}\right)^{\frac{n-\alpha}{2}} \qquad \text{for some} \,\,\, \mu>0 \,\,\, \text{and} \,\,\, x_{0}\in\mathbb{R}^{n},
\end{equation}
where the positive constant $C_{1}$ depends on $n,\alpha,c_{1},c_{2}$.

If $u\equiv0$, then we have $(u,v)\equiv(0,0)$.

From Lemma 4.1 in Dai, Fang, et al. \cite{DFHQW}, we get, for any $0<s<\frac{n}{2}$,
\begin{equation}\label{formula}
  \int_{\mathbb{R}^{n}}\frac{1}{|x-y|^{2s}}\left(\frac{1}{1+|y|^{2}}\right)^{n-s}dy=I(s)\left(\frac{1}{1+|x|^{2}}\right)^{s},
\end{equation}
where $I(s):=\frac{\pi^{\frac{n}{2}}\Gamma\left(\frac{n-2s}{2}\right)}{\Gamma(n-s)}$. If $u$ assumes the explicit form \eqref{5-2}, we can deduce from \eqref{IEM-S} and formula \eqref{formula} that
\begin{equation}\label{5-3}
  v(x)=C_{2}\left(\frac{\mu}{1+\mu^{2}|x-x_0|^2}\right)^{\frac{n-\alpha}{2}} \qquad \text{for some} \,\,\, \mu>0 \,\,\, \text{and} \,\,\, x_{0}\in\mathbb{R}^{n},
\end{equation}
where the positive constant $C_{2}$ depends on $n,\alpha,c_{1},c_{2}$. Thus $(u,v)$ must assume the explicit form \eqref{explicit}.

\emph{Sub-case (ii).} If $c_{1}(\frac{2\alpha}{n-\alpha}-q_{1})+c_{2}(\frac{n+\alpha}{n-\alpha}-q_{2})>0$, by Theorem \ref{Thm1} and \eqref{5-1}, we have $u\equiv0$, and hence $(u,v)\equiv(0,0)$.

This concludes our proof of Corollary \ref{Cor1}.

\section*{Acknowledgements}
The authors are grateful to the referees for their careful reading and valuable comments and suggestions that improved the presentation of the paper.

\end{document}